\documentclass[a4paper,11pt]{article}

\usepackage[top=0.5in, bottom=1in, left=0.8in, right=0.8in, headsep=0.in, centering]{geometry}
\usepackage[utf8]{inputenc}
\usepackage{amsmath}
\usepackage{amsfonts}
\usepackage{amssymb}
\usepackage{graphicx}
\usepackage{tikz}
\usepackage{enumitem}
\usepackage{amsthm}
\usepackage{caption} 
\usepackage{array}
\usepackage{xcolor}
\usepackage{empheq}
\usepackage{float}
\usepackage[section]{placeins}
\usepackage{subcaption}
\usepackage{bookmark}
\usepackage{hyperref}
\hypersetup{colorlinks=true,linkcolor=blue,citecolor=red,urlcolor=blue}

\newtheorem{theorem}{Theorem}[section]
\newtheorem{lemma}[theorem]{Lemma}
\newtheorem{proposition}[theorem]{Proposition}

\theoremstyle{definition}
\newtheorem{definition}[theorem]{Definition}

\newtheorem{remark}[theorem]{Remark}

\def\XXint#1#2#3{{\setbox0=\hbox{$#1{#2#3}{\int}$}
\vcenter{\hbox{$#2#3$}}\kern-.5\wd0}}

\makeatletter
\def\blfootnote{\gdef\@thefnmark{}\@footnotetext}
\makeatother

\title{American condor contracts with a cash component:\\ A variational inequality approach}

\author{%
Kiyuob Jung$^1$ \quad Jehan Oh$^{2,*}$ \quad Namgwang Woo$^2$\\[0.6em]
\small $^1$Department of Mathematics, Michigan State University\\
\small East Lansing, MI 48824, USA\\[0.3em]
\small $^2$Department of Mathematics, Kyungpook National University\\
\small Daegu 41566, Republic of Korea
}
\date{}

\begin{document}
\maketitle
\begin{abstract}
    We study a jointly exercised American condor contract whose payoff is a call condor payoff plus a positive cash amount. Its plateau and positive tails produce two inner exercise boundaries and two outer cash-exercise boundaries. We construct the unique bounded solution of the associated variational inequality, strong on each side of the plateau, and identify it with the optimal stopping value. We prove monotonicity and continuity of the four boundaries and determine their limits at expiry. Following Broadie and Detemple, we express the inner boundaries as cutoffs of the boundaries of uncapped contracts with the same cash components. We also prove call-side convexity and volatility monotonicity under a sufficient condition, solve the stationary problems explicitly and present numerical results.
\end{abstract}

\begin{NoHyper}
\blfootnote{$^*$Corresponding author.}
\blfootnote{\textit{E-mail addresses.} \texttt{kyjung@msu.edu} (K. Jung), \texttt{jehan.oh@knu.ac.kr} (J. Oh), \texttt{wng3717@knu.ac.kr} (N. Woo).}
\blfootnote{Jehan Oh is supported by the National Research Foundation of Korea (NRF) grant funded by the Korea government [Grant Nos. RS-2025-00555316 and RS-2025-25415411].}
\end{NoHyper}

\section{Introduction} \label{sec1}
We study a single American-style contract whose payoff combines a condor spread with a positive cash payment. The holder chooses one exercise time, at which all four legs are settled jointly. The resulting payoff has a maximal plateau and strictly positive tails. The positive tails create an incentive to exercise there, leading to two outer cash-exercise boundaries in addition to the two inner boundaries adjacent to the plateau. Condors are among the volatility trades examined by Chaput and Ederington \cite{chaput2005volatility}, whose Eurodollar options data show that they are traded far less often than straddles and strangles.

The connection between American options, optimal stopping and variational inequalities is developed in Bensoussan and Lions \cite{bensoussan1982applications}, Jaillet, Lamberton and Lapeyre \cite{jaillet1990variational} and Peskir and Shiryaev \cite{peskir2006optimal}. Installment options are treated in \cite{yi2008variational,yang2009valuation,yang2009variational}. Jeon and Oh \cite{jeon2019valuation} established existence and uniqueness of a strong solution and studied monotonicity and smoothness of the two exercise boundaries of an American strangle. Related pricing representations were developed by Chiarella and Ziogas \cite{chiarella2005evaluation} and Qiu \cite{qiu2016early,qiu2020american}. Ha, Jeon and Ok \cite{ha2025obstacle} study an American chooser whose obstacle is the maximum of the value functions of an American call and an American put. The fixed payoff studied here instead leads to a bounded, nonconvex obstacle with a known exercise plateau.

Let $0<K_1<k_1<k_2<K_2$ be the strike prices. The jointly settled contract comprises a long call with strike $K_1$, a long put with strike $K_2$, a short put with strike $k_1$ and a short call with strike $k_2$. Its exercise payoff is
\begin{equation} \label{payoff_V0}
    V_{0}(s) = (s-K_{1})^{+}+(K_{2}-s)^{+}-(k_{1}-s)^{+}-(s-k_{2})^{+}
    =\begin{cases}
        K_2-k_1, & 0 < s \leq K_1, \\[0.15cm]
        s-K_1+K_2-k_1, & K_1 \leq s \leq k_1, \\[0.15cm]
        K_2-K_1, & k_1 \leq s \leq k_2, \\[0.15cm]
        K_2-K_1-s+k_2, & k_2 \leq s \leq K_2, \\[0.15cm]
        k_2-K_1, & s \geq K_2.
    \end{cases}
\end{equation}
The payoff is illustrated in Figure \ref{Fig : payoff}. Its graph is a trapezoid. The payoff is maximal on the plateau $[k_1,k_2]$ and is bounded below by the positive constant $\min\{K_2-k_1,\,k_2-K_1\}$. Since
\begin{equation*}
    (K_{2}-s)^{+}-(k_{1}-s)^{+}=(s-K_{2})^{+}-(s-k_{1})^{+}+K_2-k_1,
\end{equation*}
the payoff \eqref{payoff_V0} equals the four-call combination $(s-K_1)^+-(s-k_1)^+-(s-k_2)^++(s-K_2)^+$ plus the constant $K_2-k_1$. When $k_1-K_1=K_2-k_2$, this combination is the usual call condor payoff with zero tails. We do not impose this symmetry. For a European contract, the added constant contributes a discounted cash payment at maturity. Here it is collected at the chosen exercise time and therefore changes the stopping problem. In particular, the value cannot be obtained by adding a maturity bond to the value of an American call condor.

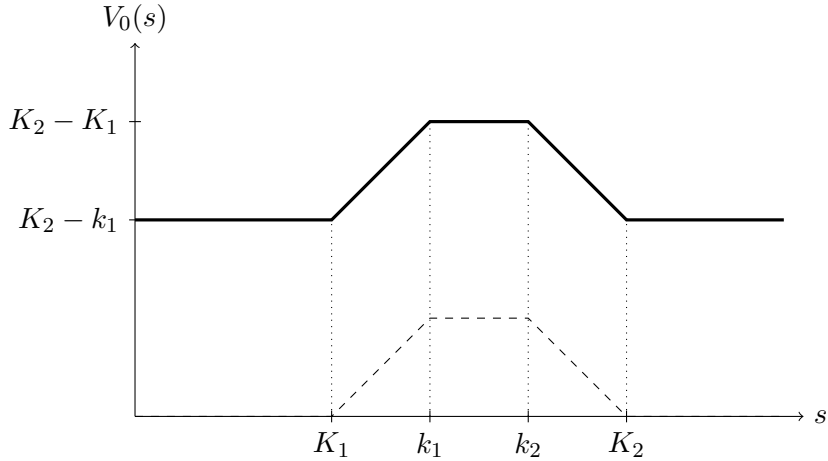
\begin{figure}[H]
\centering
\begin{tikzpicture}[scale=2.6]
    \draw[->] (0,0) -- (3.4,0) node[right] {$s$};
    \draw[->] (0,0) -- (0,1.9) node[above] {$V_0(s)$};
    \draw[very thick] (0,1) -- (1,1) -- (1.5,1.5) -- (2,1.5) -- (2.5,1) -- (3.3,1);
    \draw[dashed] (0,0) -- (1,0) -- (1.5,0.5) -- (2,0.5) -- (2.5,0) -- (3.3,0);
    \foreach \x/\l in {1/$K_1$,1.5/$k_1$,2/$k_2$,2.5/$K_2$}
        \draw (\x,0.03) -- (\x,-0.03) node[below] {\l};
    \draw (0.03,1) -- (-0.03,1) node[left] {$K_2-k_1$};
    \draw (0.03,1.5) -- (-0.03,1.5) node[left] {$K_2-K_1$};
    \draw[dotted] (1,0) -- (1,1);
    \draw[dotted] (1.5,0) -- (1.5,1.5);
    \draw[dotted] (2,0) -- (2,1.5);
    \draw[dotted] (2.5,0) -- (2.5,1);
\end{tikzpicture}
\caption{The payoff $V_0$ of the condor spread \eqref{payoff_V0} with $K_1=1$, $k_1=1.5$, $k_2=2$, $K_2=2.5$ (solid line). The dashed line is the payoff of the standard long call condor spread, which differs from $V_0$ by the constant $K_2-k_1$.}
\label{Fig : payoff}
\end{figure}

The holder may exercise at any time up to the expiration date $T>0$ and receives \eqref{payoff_V0}. We adopt the single-contract exercise convention used for the structured option strategies in \cite{chiarella2005evaluation,guan2014free,jeon2019valuation,qiu2020american} and for capped options in \cite{broadie1995american,detemple2006american}. A portfolio of separately traded American options has separate exercise rights, including those of the counterparties to its short legs, and leads to a different valuation problem.

We work in the Black--Scholes market. Under the risk-neutral measure the price $S_t$ of the underlying asset follows
\begin{equation} \label{asset_dynamics}
    \mathrm{d}S_{t}=(r-q)S_{t}\,\mathrm{d}t+\sigma S_{t}\,\mathrm{d}W_{t},
\end{equation}
where $W$ is a standard Brownian motion, the constants $\sigma>0$, $r>0$ and $q\geq0$ denote the volatility, the risk-free interest rate and the continuous dividend rate. The arbitrage-free price of the American condor spread at time $t\in[0,T]$ with $S_t=s$ is the value of the optimal stopping problem
\begin{equation} \label{optimal_stopping}
    V(s,t)=\sup_{\tau \in \mathcal{T}_{t,T}} \mathbb{E}\left[ e^{-r(\tau-t)}V_0(S_{\tau}) \,\middle|\, S_t=s \right],
\end{equation}
where $\mathcal{T}_{t,T}$ is the set of stopping times with values in $[t,T]$. The value function is expected to satisfy the variational inequality
\begin{equation} \label{condor_spread_V}
    \left\{
    \begin{array}{l}
    \displaystyle  \partial_{t} V+\frac{\sigma^{2}}{2} s^{2} \partial_{s s}V +(r-q)s \partial_{s}V -rV = 0, \quad \text{ if } V>V_{0}(s), \;\; (s,t) \in (0,\infty) \times [0,T), \\[0.3cm]
    \displaystyle  \partial_{t} V+\frac{\sigma^{2}}{2} s^{2} \partial_{s s} V+(r-q) s \partial_{s}V- rV \leq 0 , \quad \text{ if } V=V_{0}(s), \;\; (s,t) \in (0,\infty) \times [0,T),\\[0.3cm]
    \displaystyle  V(s, T)=V_{0}(s), \quad s \in (0,\infty).
    \end{array}\right.
\end{equation}
In Section \ref{sec2} we make this connection precise. We define strong solutions on each side of the plateau and a continuous, piecewise strong solution of \eqref{condor_spread_V}. Comparison on the two half lines gives uniqueness, and a verification argument stopped at the first entrance into the plateau identifies this solution with the value function \eqref{optimal_stopping}. The first hitting time of the exercise region $\{V=V_0\}$ is an optimal stopping time.

The closest precedent for the spatial reduction is Guan and Yi \cite{guan2014free}. For the American butterfly with a triangular payoff and zero tails, they use the known value at the payoff maximum to split the problem into two half-line obstacle problems. They also study boundary monotonicity, contact with the fixed boundary and convergence to the stationary problem. Here the value equals $K_2-K_1$ throughout $[k_1,k_2]$, giving the same type of reduction at the two plateau endpoints. The additional feature is the positive cash component on each side: it produces an outer exercise boundary as well as the inner boundary next to the plateau. Thus the continuation region has two components, each bounded by a cash-exercise boundary and an inner boundary that may reach the corresponding plateau endpoint. Our existence and boundary analysis builds on the variational inequality methods used for the strangle in \cite{jeon2019valuation}.

Broadie and Detemple \cite[Theorem 1]{broadie1995american} proved that the exercise boundary of an American call with a constant cap is the minimum of the cap and the ordinary uncapped call boundary. We extend this reduction to the call and put contracts with the cash components in \eqref{payoff_V0}. The corresponding inner boundaries are obtained by cutting off the uncapped boundaries at $k_1$ and $k_2$, respectively. The cash terms must remain in the uncapped problems: they change both the exercise incentives and the boundary locations. In time-to-expiry coordinates, the capped and uncapped values and outer boundaries coincide until the respective uncapped inner boundary reaches the cap. After that time we obtain a comparison for the outer boundaries, rather than a cutoff identity.

We prove monotonicity and continuity of all four boundaries and determine their expiry limits, including the zero-dividend cases. Explicit stationary solutions provide bounds uniform in the time horizon. Finally, under the sufficient condition $qk_1\geq r(K_1+k_1-K_2)$, we prove convexity of the value on the call side and deduce the corresponding volatility monotonicity of its two boundaries. The cap prevents an unconditional convexity assertion, and the volatility result is restricted to the call side. These results describe how the cash payments and the plateau affect exercise, beyond the known spatial reduction.

The rest of the paper is organized as follows. In Section \ref{sec2}, we introduce the notion of strong solution, the comparison principle and the verification theorem. In Section \ref{sec3}, we divide the problem \eqref{condor_spread_V} into two problems and show the existence and uniqueness of a $W^{2,1}_{p,\rm{loc}}$ solution for each of them, together with the corresponding uncapped problems. In Section \ref{sec4}, we define the four free boundaries, establish their monotonicity and continuity, prove the reduction to the uncapped problems and carry out a comparative static analysis with respect to the volatility. In Section \ref{sec5}, we solve the stationary problem of the American condor spread and derive bounds for the free boundaries. Section \ref{sec6} presents numerical simulations.

\section{Strong solutions, comparison principle and verification} \label{sec2}

We begin by converting the problem \eqref{condor_spread_V} into a forward parabolic problem. Letting
\begin{equation*}
    \tau = T - t,  \quad x = \ln{s}, \quad Y(x,\tau)= V(s,t),
\end{equation*}
we have $s \partial_{s}V = \partial_{x}Y$ and $s^2\partial_{ss}V = \partial_{xx}Y -\partial_{x}Y $. The problem \eqref{condor_spread_V} becomes
\begin{equation} \label{condor_spread_Y}
    \left\{
    \begin{array}{l}
    \partial_{\tau} Y-\mathcal{L}Y = 0, \quad  \text{ if } Y>Y_{0}(x),  \;\;(x,\tau) \in \mathbb{R} \times (0,T], \\[0.3cm]
    \partial_{\tau}Y-\mathcal{L}Y \geq 0 , \quad  \text{ if } Y=Y_{0}(x), \;\;(x,\tau) \in \mathbb{R} \times (0,T], \\[0.3cm]
    Y(x, 0)=Y_{0}(x), \quad x \in \mathbb{R},
    \end{array}\right.
\end{equation}
where
\begin{equation} \label{operator_L}
    \mathcal{L}Y := \frac{\sigma^2}{2}\partial_{xx}Y + \left(r-q-\frac{\sigma^2}{2} \right)\partial_{x}Y-rY
\end{equation}
and
\begin{equation*}
    Y_{0}(x) = V_0(e^x)=\left(e^x-K_{1}\right)^{+}+\left(K_{2}-e^x\right)^{+}-\left(k_{1}-e^x\right)^{+}-\left(e^x-k_{2}\right)^{+}.
\end{equation*}

We first fix the notion of strong solution used for the half-line and uncapped problems. The full condor payoff has downward jumps in its first derivative at the inner strikes. These jumps can persist for positive times, so the solution of the full problem will be defined separately, without requiring global second-order Sobolev regularity across the plateau endpoints. Let $I\subset\mathbb{R}$ be an open interval, possibly unbounded, and let $Q=I\times(0,T]$. For $1<p<\infty$, $W^{2,1}_{p}(Q)$ denotes the parabolic Sobolev space of functions $u$ with $u,\partial_x u,\partial_{xx}u,\partial_\tau u\in L^p(Q)$. For the half-line and uncapped problems, the obstacles are locally Lipschitz continuous with upward corners at finitely many points $x_1,\dots,x_m$ of $\overline{I}$. Their initial corners may prevent $W^{2,1}_p$ regularity near $(x_j,0)$ for large $p$. We therefore write
\begin{equation*}
    \mathcal{W}_p(Q) := \left\{ u \in C(\overline{I}\times[0,T]) : u \in W^{2,1}_{p}\left( Q_{R,\rho} \right) \text{ for all } R,\rho>0 \right\},
\end{equation*}
where
\begin{equation*}
    Q_{R,\rho}:=\left( (I\cap(-R,R)) \times (0,T) \right) \setminus \bigcup_{j=1}^{m} B_{\rho}(x_j,0)
\end{equation*}
and
$B_\rho(x_j,0)$ is the open disc of radius $\rho$ centered at the corner point $(x_j,0)$. Functions in $\mathcal{W}_p(Q)$ belong to $W^{2,1}_{p,\rm{loc}}(Q)$, and for $p>3$ they are continuously differentiable in $x$ on $\overline{I}\times(0,T]$ by the Sobolev embedding theorem. Their derivatives are defined almost everywhere and, for $u\in\mathcal{W}_p(Q)$ and a smooth function $\psi$, we have $\partial_x u=\psi'$, $\partial_{xx}u=\psi''$ and $\partial_\tau u=0$ almost everywhere on the coincidence set $\{u=\psi\}$, see \cite[Chapter 1]{friedman2012variational}.

\begin{definition} \label{def_solution}
    Let $\psi\in C(\overline{I})$ be an obstacle and let $g$ be a continuous function on $\partial I\times[0,T]$ with $g=\psi$ at $\tau=0$. A function $Y\in\mathcal{W}_p(Q)$ is a strong solution of the obstacle problem
    \begin{equation} \label{general_VI}
        \left\{
        \begin{array}{l}
        \partial_{\tau} Y-\mathcal{L}Y = 0, \quad  \text{ if } Y>\psi(x),  \;\;(x,\tau) \in Q, \\[0.3cm]
        \partial_{\tau}Y-\mathcal{L}Y \geq 0 , \quad  \text{ if } Y=\psi(x), \;\;(x,\tau) \in Q, \\[0.3cm]
        Y=g \;\; \text{ on } \partial I\times[0,T], \qquad Y(x, 0)=\psi(x), \quad x \in I,
        \end{array}\right.
    \end{equation}
    if $Y\geq\psi$ in $Q$, if $\partial_\tau Y-\mathcal{L}Y\geq0$ almost everywhere in $Q$, if $\partial_\tau Y-\mathcal{L}Y=0$ almost everywhere on $\{Y>\psi\}$, and if $Y$ takes the initial and boundary values in \eqref{general_VI} pointwise. A function $w\in\mathcal{W}_p(Q)$ is a supersolution of \eqref{general_VI} if $w\geq\psi$ in $Q$, if $\partial_\tau w-\mathcal{L}w\geq0$ almost everywhere in $Q$ and if $w\geq g$ on $\partial I\times[0,T]$. A function $w\in\mathcal{W}_p(Q)$ is a subsolution of \eqref{general_VI} if $\partial_\tau w-\mathcal{L}w\leq0$ almost everywhere on $\{w>\psi\}$, if $w\leq g$ on $\partial I\times[0,T]$ and if $w(\cdot,0)\leq\psi$ in $I$.
\end{definition}
The condition at $\tau=0$ for a supersolution follows from $w\geq\psi$. A strong solution is both a supersolution and a subsolution. When $I=\mathbb{R}$ there is no boundary condition, and when $I$ is a half line the boundary condition is imposed at its finite endpoint only. We define the exercise region and the continuation region of a strong solution $Y$ by
\begin{align*}
    &\mathcal{E} = \left\{(x,\tau) \in Q: Y(x,\tau) = \psi(x)\right\}, \\[0.3cm]
    &\mathcal{C} = \left\{(x,\tau) \in Q: Y(x,\tau) > \psi(x)\right\}.
\end{align*}

The following maximum principle is the main tool of the paper. It requires only interior Sobolev regularity, so it also applies to differences of spatial translates whose initial corners occur at different points. All the solutions we consider are bounded or grow at most like $e^{x}$.
\begin{lemma} \label{lem_max_principle}
    Let $1<p<\infty$ and let $w\in C(\overline I\times[0,T])\cap W^{2,1}_{p,\mathrm{loc}}(I\times(0,T))$ satisfy $w(x,\tau)\leq C e^{|x|}$ in $Q$ for some constant $C$. Suppose that $\partial_\tau w-\mathcal{L}w\leq0$ almost everywhere on $\{w>0\}$, that $w\leq0$ on $\partial I\times[0,T]$ and that $w(\cdot,0)\leq0$ in $I$. Then $w\leq0$ in $Q$.
\end{lemma}
\begin{proof}
    Let $\lambda := 2\sigma^{2}+2\left|r-q-\frac{\sigma^2}{2}\right|$ and $h(x,\tau):=e^{\lambda\tau}\left(e^{2x}+e^{-2x}\right)$. A direct computation gives
    \begin{equation*}
        \partial_\tau h-\mathcal{L}h = e^{\lambda\tau}\left[\left(\lambda-2\sigma^2-2\left(r-q-\frac{\sigma^2}{2}\right)+r\right)e^{2x}+\left(\lambda-2\sigma^2+2\left(r-q-\frac{\sigma^2}{2}\right)+r\right)e^{-2x}\right]\geq0.
    \end{equation*}
    For $\delta>0$ put $v:=w-\delta h$. Then $(\partial_\tau-\mathcal L)v\leq0$ almost everywhere on $\{v>0\}$. The upper growth bound and $h\geq e^{2|x|}$ imply that $v<0$ for $|x|$ sufficiently large, uniformly in time. Also $v<0$ on the initial and finite lateral boundaries. Fix $T'<T$ and $a>0$. The closure of
    \begin{equation*}
        D_a:=\{(x,\tau):0<\tau\leq T',\ v(x,\tau)>a\}
    \end{equation*}
    is therefore compact and stays away from the initial and lateral boundaries. In a neighborhood of this closure, $v$ belongs to $W^{2,1}_p$ and is positive. Convolve $v$ there with a nonnegative smooth space-time mollifier. Since $\mathcal L$ has constant coefficients, the resulting smooth functions $v_\varepsilon$ satisfy $(\partial_\tau-\mathcal L)v_\varepsilon\leq0$ in $D_a$ for sufficiently small $\varepsilon$, and converge uniformly to $v$ on its closure. On the parabolic boundary of $D_a$ we have $v_\varepsilon\leq a+o(1)$. The classical parabolic maximum principle gives $v_\varepsilon\leq a+o(1)$ throughout $D_a$, and passing to the limit contradicts the definition of $D_a$ unless $D_a$ is empty. Since $a>0$ and $T'<T$ are arbitrary, continuity gives $v\leq0$ on $Q$. Finally let $\delta\searrow0$. This constant-coefficient argument applies to every $p>1$.
\end{proof}

\begin{lemma} \label{lem_comparison}
    Let $Y$ be a strong solution of \eqref{general_VI} with $|Y|\leq Ce^{|x|}$ in $Q$, and let $w\in\mathcal{W}_p(Q)$ satisfy $|w|\leq Ce^{|x|}$ in $Q$.
    \begin{enumerate}[label=(\roman*)]
        \item If $w$ is a supersolution of \eqref{general_VI}, then $Y\leq w$ in $Q$.
        \item If $w$ is a subsolution of \eqref{general_VI}, then $Y\geq w$ in $Q$.
        \item Let $\beta:\mathbb{R}\to\mathbb{R}$ be nondecreasing and let $\varphi\in C(\overline{I})$. If $u_1,u_2\in\mathcal{W}_p(Q)$ satisfy $|u_i|\leq Ce^{|x|}$,
        \begin{equation*}
            \partial_\tau u_1-\mathcal{L}u_1+\beta(u_1-\varphi)\leq \partial_\tau u_2-\mathcal{L}u_2+\beta(u_2-\varphi) \quad \text{ almost everywhere in } Q,
        \end{equation*}
        and $u_1\leq u_2$ on the parabolic boundary of $Q$, then $u_1\leq u_2$ in $Q$.
    \end{enumerate}
    In particular, the strong solution of \eqref{general_VI} with growth at most $Ce^{|x|}$ is unique.
\end{lemma}
\begin{proof}
    For (i), set $U:=Y-w$. On the set $\{U>0\}$ we have $Y>w\geq\psi$, so $\partial_\tau Y-\mathcal{L}Y=0$ almost everywhere there and hence $\partial_\tau U-\mathcal{L}U=-(\partial_\tau w-\mathcal{L}w)\leq0$ almost everywhere on $\{U>0\}$. Moreover $U\leq0$ on the parabolic boundary of $Q$. Lemma \ref{lem_max_principle} gives $U\leq0$. For (ii), set $U:=w-Y$. On $\{U>0\}$ we have $w>Y\geq\psi$, so $\partial_\tau w-\mathcal{L}w\leq0$ almost everywhere there, while $\partial_\tau Y-\mathcal{L}Y\geq0$ almost everywhere in $Q$. Hence $\partial_\tau U-\mathcal{L}U\leq0$ almost everywhere on $\{U>0\}$ and $U\leq0$ on the parabolic boundary, and Lemma \ref{lem_max_principle} gives $U\leq0$. For (iii), set $U:=u_1-u_2$. On $\{U>0\}$ we have $\beta(u_1-\varphi)\geq\beta(u_2-\varphi)$ almost everywhere because $\beta$ is nondecreasing, so that $\partial_\tau U-\mathcal{L}U\leq0$ almost everywhere on $\{U>0\}$, and Lemma \ref{lem_max_principle} applies again. The uniqueness follows from (i) or (ii), since a strong solution is both a supersolution and a subsolution.
\end{proof}

We now specify the solution class for the full condor problem. Write $M:=K_2-K_1$, $\ell_1:=\ln k_1$, $\ell_2:=\ln k_2$ and $I_-:=(-\infty,\ell_1)$, $I_+:=(\ell_2,\infty)$. The payoff is at most $M$ and equals $M$ on $[k_1,k_2]$. Since $r>0$, the optimal stopping value in \eqref{optimal_stopping} is at most $M$, and immediate exercise is optimal on this plateau. This observation uses no regularity of the value function.

For a fixed $p>1$, a \emph{piecewise strong solution} of \eqref{condor_spread_Y} means a bounded function $Y\in C(\mathbb R\times[0,T])$ such that $Y=M$ on $[\ell_1,\ell_2]\times[0,T]$ and, on each $I_\pm\times(0,T]$, its restriction is a strong solution in the sense of Definition \ref{def_solution}, with obstacle $Y_0|_{I_\pm}$, initial value $Y_0$ and constant boundary value $M$ at the finite endpoint. The excluded initial corner is $(\ln K_1,0)$ on $I_-$ and $(\ln K_2,0)$ on $I_+$. No matching of the one-sided spatial derivatives at $\ell_1$ or $\ell_2$ is imposed. In particular, this definition does not assert $Y\in\mathcal W_p(\mathbb R\times(0,T])$. The differential conditions in \eqref{condor_spread_Y} hold almost everywhere on the three open spatial intervals, and the global distributional inequality will follow from the construction in Theorem \ref{thm_full_existence}.

We close this section with the verification theorem. A piecewise strong solution of \eqref{condor_spread_V} means $V(s,t)=Y(\ln s,T-t)$ for such a $Y$. The exercise and continuation regions in the original variables are
\begin{equation*}
    \mathcal{E}_V=\left\{(s,t)\in(0,\infty)\times[0,T] : V(s,t)=V_0(s)\right\}, \qquad \mathcal{C}_V=\left\{(s,t)\in(0,\infty)\times[0,T] : V(s,t)>V_0(s)\right\}.
\end{equation*}
\begin{theorem} \label{thm_verification}
    Let $p>3$ and let $V$ be a piecewise strong solution of \eqref{condor_spread_V}. Then $V$ coincides with the value function \eqref{optimal_stopping}. Moreover, for every $(s,t)\in(0,\infty)\times[0,T)$ the stopping time
    \begin{equation*}
        \tau^{*}:=\inf\left\{u\in[t,T] : (S_u,u)\in\mathcal{E}_V\right\}
    \end{equation*}
    is optimal for \eqref{optimal_stopping}.
\end{theorem}
\begin{proof}
    Fix $(s,t)$ with $t<T$ and let $S$ solve \eqref{asset_dynamics} with $S_t=s$. If $s\in[k_1,k_2]$, then $V(s,t)=M$ and immediate exercise is optimal, since all discounted payoffs are at most $M$. Suppose that $s$ lies outside the plateau and let
    \begin{equation*}
        \eta:=\inf\{u\geq t:S_u\in[k_1,k_2]\},\qquad \inf\emptyset:=+\infty.
    \end{equation*}
    For every $\theta\in\mathcal T_{t,T}$, continuity of $S$ and $V_0\leq M$ imply the pathwise inequality
    \begin{equation*}
        e^{-r(\theta-t)}V_0(S_\theta)
        \leq e^{-r((\theta\wedge\eta)-t)}V_0(S_{\theta\wedge\eta}).
    \end{equation*}
    Thus the supremum in \eqref{optimal_stopping} is unchanged if we restrict to stopping times $\theta\leq\eta\wedge T$.

    Let $\mathcal A V:=\partial_tV+\frac{\sigma^2}{2}s^2\partial_{ss}V+(r-q)s\partial_sV-rV$. On the side of the plateau containing $s$, the definition gives $\mathcal AV\leq0$ almost everywhere, with equality almost everywhere on $\mathcal C_V$. The logarithmic change of variables preserves local $W^{2,1}_p$ regularity on spatial intervals bounded away from zero. To localize entirely within this side, let $J_m=(1/m,k_1-1/m)$ if $s<k_1$, and $J_m=(k_2+1/m,m)$ if $s>k_2$, taking $m$ sufficiently large that $s\in J_m$. Put $\zeta_m:=\inf\{u\geq t:S_u\notin J_m\}$. For $t<b<T$ and $\theta\leq\eta\wedge T$, set $\theta_{m,b}:=\theta\wedge b\wedge\zeta_m$. On this compact spatial localization, before terminal time, the diffusion is uniformly nondegenerate and the generalized It\^{o} formula \cite[Theorem 2.10.1]{krylov1980controlled} yields
    \begin{equation} \label{ito_verification}
        \mathbb E\!\left[e^{-r(\theta_{m,b}-t)}V(S_{\theta_{m,b}},\theta_{m,b})\right]
        =V(s,t)+\mathbb E\!\left[\int_t^{\theta_{m,b}}e^{-r(u-t)}\mathcal AV(S_u,u)\,\mathrm du\right]
        \leq V(s,t).
    \end{equation}
    The stochastic integral has zero expectation because $p>3$ gives a bounded spatial derivative on each such localization. If $t=0$, a Sobolev extension across the initial calendar-time boundary permits the same application of It\^{o}'s formula. The paths of $S$ have a positive minimum and a finite maximum on bounded time intervals, so $\theta_{m,b}\to\theta\wedge b$ as $m\to\infty$. Boundedness and continuity of $V$, including its boundary value $M$ at the plateau, allow dominated convergence, first as $m\to\infty$ and then as $b\nearrow T$. Since $V\geq V_0$, we obtain
    \begin{equation*}
        \mathbb E\!\left[e^{-r(\theta-t)}V_0(S_\theta)\right]\leq V(s,t).
    \end{equation*}
    Together with the preceding restriction of the class of stopping times, this shows that $V$ dominates the value function.

    Conversely, the contact set is closed and contains the plateau and the terminal slice $\{t=T\}$. Hence $\tau^*$ is a stopping time, $\tau^*\leq\eta\wedge T$, and $V(S_{\tau^*},\tau^*)=V_0(S_{\tau^*})$. Before $\tau^*$ the process stays in $\mathcal C_V$, where $\mathcal AV=0$ almost everywhere. Since $S$ has a transition density, its expected occupation time in any space-time Lebesgue null set is zero. Consequently \eqref{ito_verification}, with $\theta=\tau^*$, holds with equality. Passing to the same limits gives
    \begin{equation*}
        V(s,t)=\mathbb E\!\left[e^{-r(\tau^*-t)}V_0(S_{\tau^*})\right].
    \end{equation*}
    Thus $V$ is the value function and $\tau^*$ is optimal. At $t=T$, the assertion follows from the terminal condition. The argument never applies It\^{o}'s formula across an inner strike.
\end{proof}

\section{The existence and uniqueness of a solution} \label{sec3}

Throughout this section $1<p<\infty$ is fixed. We first record the elementary stopping argument that determines the boundary data at the plateau.
\begin{proposition} \label{prop_K_2-K_1}
    Let $V$ be the optimal stopping value in \eqref{optimal_stopping}. Then $V_0(s)\leq V(s,t)\leq K_2-K_1$, and
    \begin{equation*}
        V(s,t)=K_2-K_1,\qquad (s,t)\in[k_1,k_2]\times[0,T].
    \end{equation*}
\end{proposition}
\begin{proof}
    Immediate exercise gives $V\geq V_0$. Every admissible discounted payoff is at most $\max V_0=K_2-K_1$, since $r>0$. On the plateau the immediate payoff already equals this bound.
\end{proof}

Accordingly, the piecewise formulation of \eqref{condor_spread_Y} consists of the following two strong obstacle problems, with the fixed plateau value between them:
\begin{equation} \label{condor_spread_left}
   \begin{cases}
    \; \partial_{\tau} Y-\mathcal{L} Y = 0, \quad
    \text{ if } Y > \psi_1(x),
    & (x, \tau) \in (-\infty, \ln k_{1} ) \times (0,T] ,\\[0.3cm]
    \; \partial_{\tau} Y-\mathcal{L} Y \geq 0, \quad \text{ if } Y =\psi_1(x), & (x,\tau) \in (-\infty,\ln k_{1}) \times (0,T] ,\\[0.3cm]
    \; Y(x, 0)=\psi_1(x),& x \in (-\infty, \ln k_{1}),\\[0.3cm]
    \; Y\left(\ln k_{1}, \tau\right)=K_{2}-K_{1},& \tau \in [0,T],
   \end{cases}
\end{equation}
\begin{equation} \label{condor_spread_right}
    \begin{cases}
    \partial_{\tau} Y-\mathcal{L} Y = 0, \quad
    \text{ if } Y > \psi_2(x),& (x,\tau) \in (\ln k_{2}, \infty) \times (0,T] ,\\[0.3cm]
    \partial_{\tau} Y-\mathcal{L} Y \geq 0, \quad
    \text{ if } Y =\psi_2(x), & (x,\tau) \in (\ln k_{2}, \infty) \times (0,T] ,\\[0.3cm]
    Y(x, 0)=\psi_2(x), & x \in (\ln k_{2}, \infty), \\[0.3cm]
    Y\left(\ln k_{2}, \tau\right)=K_{2}-K_{1}, & \tau \in [0,T],
    \end{cases}
\end{equation}
where
\begin{equation} \label{def_psi12}
    \psi_1(x):=\left(e^{x}-K_{1}\right)^{+}+K_{2}-k_{1}, \qquad \psi_2(x):=\left(K_{2}-e^x\right)^{+}+k_{2}-K_{1}.
\end{equation}
Note that $\psi_1=Y_0$ on $(-\infty,\ln k_1]$ and $\psi_2=Y_0$ on $[\ln k_2,\infty)$, and that $\psi_1(\ln k_1)=\psi_2(\ln k_2)=K_2-K_1$. The problem \eqref{condor_spread_left} is the problem of an American call option with strike $K_1$ plus the cash amount $K_2-k_1$, whose value is known to be $K_2-K_1$ at the level $k_1$. The problem \eqref{condor_spread_right} is the problem of an American put option with strike $K_2$ plus the cash amount $k_2-K_1$ with the analogous property at the level $k_2$. Strong solutions of \eqref{condor_spread_left} and \eqref{condor_spread_right} are understood in the sense of Definition \ref{def_solution} with $I=(-\infty,\ln k_1)$ and $I=(\ln k_2,\infty)$, respectively, and the corner points are $(\ln K_1,0)$ and $(\ln K_2,0)$. By definition, a piecewise strong solution of \eqref{condor_spread_Y} restricts to strong solutions of \eqref{condor_spread_left} and \eqref{condor_spread_right}. Conversely, bounded strong solutions of these two problems can be joined continuously to the constant plateau value, without requiring their spatial derivatives to match the zero derivative on the plateau. Since the two problems are treated in the same way, we give the details for \eqref{condor_spread_left} only.

To prove the existence of a solution to \eqref{condor_spread_left}, we consider the truncated problems
\begin{equation} \label{condor_spread_Yn}
\left\{
\begin{array}{l}
\partial_{\tau} Y_{n}-\mathcal{L} Y_{n}=0, \text { if } Y_{n}>\psi_1(x),\quad (x, \tau ) \in(-n,\ln k_{1}) \times (0,T], \\[0.3cm]
\partial_{\tau} Y_{n}-\mathcal{L} Y_{n} \geq 0, \text { if } Y_{n}=\psi_1(x),\quad (x, \tau ) \in (-n,\ln k_{1}) \times (0,T], \\[0.3cm]
Y_{n}(-n, \tau)=K_2-k_1, \quad Y_{n} (\ln k_{1},\tau) =K_{2}-K_{1} , \quad \tau \in [0, T], \\[0.3cm]
Y_{n}(x,0)=\psi_1(x), \quad x \in [-n, \ln k_{1}],
\end{array}\right.
\end{equation}
where $ n \in \mathbb{N} $ with $ n > \ln{\frac{2}{K_{1}}}$. The condition on $n$ guarantees that $e^{-n}<K_1/2$, so that the artificial boundary $x=-n$ lies in the region where the obstacle $\psi_1$ is constant. At $x=-n$ we impose the value $\psi_1(-n)=K_2-k_1$ of the obstacle. This is the natural choice, because the solution of \eqref{condor_spread_left} coincides with the obstacle for $x$ close to $-\infty$, see Proposition \ref{prop_YV} below, and it makes the boundary data compatible with the initial data at the corners $(-n,0)$ and $(\ln k_1,0)$. We write $Q_n:=(-n,\ln k_1)\times(0,T]$.

\begin{lemma} \label{lemma_Yn}
    For any $n \in \mathbb{N}$ with $n> \ln{\frac{2}{K_{1}}}$, there exists a unique strong solution $Y_{n} \in \mathcal{W}_p(Q_n)$ of the problem \eqref{condor_spread_Yn}, and it satisfies
    \begin{equation} \label{lemma_Yn_prop1}
        \psi_1(x) \leq Y_{n}(x,\tau) \leq K_{2}-K_{1}, \qquad (x,\tau)\in\overline{Q_n}.
    \end{equation}
    Moreover, for every $\rho>0$ there is a constant $C=C(n,\rho,p)$ such that
    \begin{equation} \label{lemma_Yn_W21p}
        \left\|Y_{n}\right\|_{W_{p}^{2,1}\left(Q_n \setminus B_{\rho}(\ln K_1,0)\right)} \leq C.
    \end{equation}
\end{lemma}
\begin{proof}
    The uniqueness follows from Lemma \ref{lem_comparison}. For the existence we use a penalty method. Let $C_0:=(r+q)k_1+r(1+K_2)$ and define a family of penalty functions $ \beta_{\varepsilon} \in C^{\infty}(\mathbb{R}) $, $ 0 < \varepsilon < 1 $, satisfying
    \begin{equation} \label{penalty_ft}
        \left\{
        \begin{array}{l}
            \beta_{\varepsilon}(t) \leq 0, \quad \beta_{\varepsilon}^{\prime}(t) \geq 0, \quad \beta_{\varepsilon}^{\prime \prime}(t) \leq 0, \quad \forall t \in \mathbb{R}, \\[0.3cm]
            \beta_{\varepsilon}(t)=0 \;\; \text { if } \;\; t \geq \varepsilon, \qquad \beta_{\varepsilon}(0)=-C_{0}, \\[0.3cm]
            \displaystyle \lim_{\varepsilon \searrow 0} \beta_{\varepsilon}(t)=0 \quad \text { if } \quad t>0, \qquad \displaystyle \lim_{\varepsilon \searrow 0} \beta_{\varepsilon}(t)=-\infty \quad \text { if } \quad t<0 .
        \end{array} \right.
    \end{equation}
    We also define a family of functions $\varphi_{\varepsilon} \in C^{\infty}(\mathbb{R}) $, $ 0<\varepsilon<1 $, satisfying
    \begin{eqnarray} \label{psi_epsilon}
            \left\{
        \begin{array}{l}
            \varphi_{\varepsilon}(t) \geq 0, \quad 0 \leq \varphi_{\varepsilon}^{\prime}(t) \leq 1, \quad \varphi_{\varepsilon}^{\prime \prime}(t) \geq 0, \quad \forall t \in \mathbb{R}, \\[0.3cm]
            \varphi_{\varepsilon}(t)=t \quad \text { if } \quad t \geq \varepsilon, \qquad \varphi_{\varepsilon}(t)=0 \quad \text { if } \quad t \leq-\varepsilon, \\[0.3cm]
            \displaystyle \lim _{\varepsilon \searrow 0} \varphi_{\varepsilon}(t)=t^{+}, \quad \forall t \in \mathbb{R} ,
        \end{array}\right.
    \end{eqnarray}
    so that
    \begin{equation} \label{prop1_varphi_pf}
        t^{+} \leq \varphi_{\varepsilon}(t) \leq (t+\varepsilon)^{+}, \quad \forall t \in \mathbb{R}.
    \end{equation}
    We write $\psi_{1,\varepsilon}(x):=\varphi_{\varepsilon}(e^x-K_1)+K_2-k_1$, which is a smooth function converging uniformly to $\psi_1$ as $\varepsilon\searrow0$, and we always take $0<\varepsilon<\min\{K_1/2,\,k_1-K_1\}$. Then $\psi_{1,\varepsilon}(-n)=K_2-k_1$ and $\psi_{1,\varepsilon}(\ln k_1)=K_2-K_1$ by the choice of $n$ and $\varepsilon$. The penalized problem is
    \begin{equation} \label{condor_spread_Yne}
            \left\{
        \begin{array}{l}
            \partial_{\tau} Y_{n, \varepsilon}-\mathcal{L} Y_{n, \varepsilon} +\beta_{\varepsilon} \left(Y_{n, \varepsilon} -\psi_{1,\varepsilon}(x)\right)=0
            , \quad (x, \tau) \in Q_n, \\[0.3cm]
            Y_{n, \varepsilon}(-n, \tau)=K_{2}-k_{1}, \quad Y_{n, \varepsilon}(\ln k_{1}, \tau)=K_{2}-K_{1}, \quad \tau \in[0, T], \\[0.3cm]
            Y_{n, \varepsilon}(x,0)=\psi_{1,\varepsilon}(x), \quad x \in [-n, \ln k_{1}].
        \end{array}\right.
    \end{equation}

    We first show that \eqref{condor_spread_Yne} has a solution $Y_{n,\varepsilon}\in W^{2,1}_p(Q_n)$ for every $1<p<\infty$ by the Leray--Schauder fixed point theorem in the form of Schaefer, see \cite[Section 9.2.2, Theorem 4]{evans2010partial}. Let $ B:=C(\overline{Q_n}) $ with the supremum norm. For $ w \in B $ let $ Fw $ be the solution of the linear problem
    \begin{equation} \label{penalty_problem}
        \left\{
        \begin{array}{l}
        \partial_{\tau} u-\mathcal{L}u=
        -\beta_{\varepsilon}\left(w-\psi_{1,\varepsilon}(x)\right), \quad (x, \tau) \in Q_n, \\[0.3cm]
        u(-n, \tau)=K_2-k_1, \quad u(\ln k_1, \tau)=K_2-K_1, \quad \tau \in[0, T], \\[0.3cm]
        u(x,0)=\psi_{1,\varepsilon}(x), \quad x \in [-n, \ln k_1].
        \end{array}\right.
    \end{equation}
    The right-hand side of \eqref{penalty_problem} is bounded, the initial datum is smooth and the boundary data are constants which agree with the initial datum at the two corners. Hence, by \cite[Chapter IV, Theorem 9.1]{ladyzhenskaia1988linear}, the problem \eqref{penalty_problem} has a unique solution $u\in W^{2,1}_p(Q_n)$ for every $1<p<\infty$, and
    \begin{equation} \label{penalty_W21p_estimate}
        \|Fw\|_{W^{2,1}_p(Q_n)}\leq C\left(\left\|\beta_{\varepsilon}\left(w-\psi_{1,\varepsilon}\right)\right\|_{L^\infty(Q_n)}+1\right)
    \end{equation}
    with a constant $C=C(n,p,T,\varepsilon)$, also depending on the fixed model parameters. Here $\varepsilon$ is fixed, and its appearance accounts for the trace norm of the smooth initial datum. Uniform estimates as $\varepsilon\searrow0$ are obtained separately below. Since $\beta_{\varepsilon}$ is continuous, the right-hand side of \eqref{penalty_W21p_estimate} is bounded on bounded subsets of $B$, and for $p>3$ the embedding $W^{2,1}_p(Q_n)\subset C^{\alpha,\alpha/2}(\overline{Q_n})$ with $\alpha=1-3/p$ is compact. Therefore $F$ maps bounded subsets of $B$ into precompact subsets of $B$. If $ w_j \to w$ in $B$, then $ Fw_j-Fw$ vanishes on the parabolic boundary of $Q_n$ and satisfies
    \begin{equation*}
        \left|(\partial_{\tau}-\mathcal{L})(Fw_j-Fw)\right| = \left|\beta_{\varepsilon}(w_j-\psi_{1,\varepsilon}) -\beta_{\varepsilon}(w-\psi_{1,\varepsilon})\right| \leq \Lambda \|w_j-w\|_{L^\infty(Q_n)},
    \end{equation*}
    where $\Lambda$ is a bound for $\beta_{\varepsilon}^{\prime}$ on the bounded range of the arguments. Comparing $Fw_j-Fw$ with the functions $\pm\Lambda\|w_j-w\|_{L^\infty(Q_n)}\tau$ by Lemma \ref{lem_max_principle}, we get
    \begin{equation*}
        \|Fw_j-Fw \|_{L^{\infty}(Q_n)} \leq \Lambda T \| w_j-w \|_{L^{\infty}(Q_n)},
    \end{equation*}
    so $F:B\to B$ is continuous and compact. It remains to show that the set of all $w\in B$ with $w=\lambda Fw$ for some $\lambda\in[0,1]$ is bounded. Such a $w$ belongs to $W^{2,1}_p(Q_n)$ and solves
    \begin{equation*}
        \partial_{\tau} w-\mathcal{L}w=-\lambda\beta_{\varepsilon}\left(w-\psi_{1,\varepsilon}(x)\right) \geq 0 \quad \text{ in } Q_n
    \end{equation*}
    with the boundary and initial data of \eqref{penalty_problem} multiplied by $\lambda$, which are nonnegative. Lemma \ref{lem_max_principle} applied to $-w$ gives $w\geq0$. Consequently $w-\psi_{1,\varepsilon}\geq-\psi_{1,\varepsilon}(\ln k_1)=-(K_2-K_1)$ in $Q_n$, and since $\beta_{\varepsilon}$ is nondecreasing,
    \begin{equation*}
        0\leq-\lambda\beta_{\varepsilon}\left(w-\psi_{1,\varepsilon}\right)\leq -\beta_{\varepsilon}(-(K_2-K_1))=:C_1(\varepsilon).
    \end{equation*}
    Comparing $w$ with $K_2-K_1+C_1(\varepsilon)\tau$ by Lemma \ref{lem_max_principle}, we obtain $0\leq w\leq K_2-K_1+C_1(\varepsilon)T$. Schaefer's theorem now yields a fixed point $Y_{n,\varepsilon}=FY_{n,\varepsilon}\in W^{2,1}_p(Q_n)$, which is a solution of \eqref{condor_spread_Yne}. It is unique by Lemma \ref{lem_comparison} (iii).

    Next, we claim that
    \begin{equation} \label{prop1_varphi}
         \psi_{1,\varepsilon}(x) \leq Y_{n, \varepsilon}(x,\tau) \leq K_{2}-K_{1}+C_0\tau, \qquad (x,\tau)\in\overline{Q_n}.
    \end{equation}
    Using $\varphi_{\varepsilon}^{\prime \prime}\geq 0$, $0 \leq \varphi_{\varepsilon}^{\prime} \leq 1 $, \eqref{prop1_varphi_pf} and $e^x\leq k_1$ in $Q_n$, we compute
    \begin{align*}
        & \partial_{\tau}\psi_{1,\varepsilon}-\mathcal{L}\psi_{1,\varepsilon}+  \beta_{\varepsilon}\left(\psi_{1,\varepsilon}-\psi_{1,\varepsilon}\right) \\[0.3cm]
        & =-\frac{\sigma^{2}}{2} \varphi_{\varepsilon}^{\prime \prime}\left(e^{x}-K_{1}\right)e^{2x} -(r-q)\varphi_{\varepsilon}^{\prime} \left(e^{x}-K_{1}\right)e^{x} + r\left[\varphi_{\varepsilon}\left(e^{x}-K_{1}\right)+K_{2}-k_{1}\right]-C_{0} \\[0.3cm]
        & \leq q e^{x} +r\left(e^{x}-K_{1}+\varepsilon\right)^{+}+ r(K_{2}-k_{1})-C_{0} \\[0.3cm]
        & \leq (r+q)k_1 +r(1+K_{2}) -C_{0} = 0 .
    \end{align*}
    Hence $\psi_{1,\varepsilon}$ is a subsolution of the penalized equation, and it coincides with $Y_{n,\varepsilon}$ on the parabolic boundary of $Q_n$. Lemma \ref{lem_comparison} (iii) gives $\psi_{1,\varepsilon}\leq Y_{n,\varepsilon}$. In particular the argument of $\beta_\varepsilon$ in \eqref{condor_spread_Yne} is nonnegative, and the definition of $\beta_{\varepsilon}$ gives
    \begin{equation} \label{penalty_bounded}
        -C_{0} \leq \beta_{\varepsilon}\left(Y_{n, \varepsilon} - \psi_{1,\varepsilon}\right) \leq 0 \quad \text{ in } Q_n.
    \end{equation}
    Consequently $\partial_\tau Y_{n,\varepsilon}-\mathcal{L}Y_{n,\varepsilon}\leq C_0$, while $(\partial_\tau-\mathcal{L})(K_2-K_1+C_0\tau)=C_0+r(K_2-K_1+C_0\tau)\geq C_0$ and $K_2-K_1+C_0\tau\geq Y_{n,\varepsilon}$ on the parabolic boundary of $Q_n$. Lemma \ref{lem_max_principle} applied to $Y_{n,\varepsilon}-(K_2-K_1+C_0\tau)$ gives the upper bound in \eqref{prop1_varphi}.

    By \eqref{penalty_bounded}, $Y_{n,\varepsilon}$ solves a linear equation with a right-hand side bounded by $C_0$, with constant boundary data and with initial data $\psi_{1,\varepsilon}$, which are Lipschitz continuous with constant $k_1$ uniformly in $\varepsilon$. The H\"{o}lder estimate up to the boundary, see \cite[Theorem 6.33]{lieberman1996second} and \cite[Chapter III]{ladyzhenskaia1988linear}, yields
    \begin{equation} \label{Yne_Calpha}
        \left\|Y_{n, \varepsilon}\right\|_{C^{\alpha, \alpha / 2}\left(\overline{Q_n}\right)} \leq C
    \end{equation}
    for some $\alpha\in(0,1)$ and some constant $ C $ independent of $\varepsilon$. Moreover, for $\varepsilon<K_1(1-e^{-\rho})$ the initial data $\psi_{1,\varepsilon}$ coincide with $\psi_1$ outside $(\ln K_1-\rho,\ln K_1+\rho)$, and $\psi_1$ is smooth there. The local $W^{2,1}_p$ estimates, see \cite[Chapter IV, Section 10]{ladyzhenskaia1988linear} and \cite[Theorem 7.22]{lieberman1996second}, together with \eqref{penalty_bounded} give
    \begin{equation} \label{Yne_W21p}
        \left\|Y_{n, \varepsilon}\right\|_{W_{p}^{2,1}\left(Q_n\setminus B_{\rho}(\ln K_1,0)\right)} \leq C(n,\rho,p)
    \end{equation}
    uniformly in $\varepsilon$. By \eqref{Yne_Calpha} and \eqref{Yne_W21p}, there are a sequence $\varepsilon_j\searrow0$ and a function $Y_n\in\mathcal{W}_p(Q_n)$ such that
    \begin{equation} \label{lemma_conv_w}
        Y_{n, \varepsilon_j} \rightharpoonup Y_{n} \quad \text { weakly in }  W_{p}^{2,1}\left(Q_n \setminus B_{\rho}(\ln K_1,0) \right) \text{ for every } \rho>0
    \end{equation}
    and
    \begin{equation} \label{lemma_conv_c}
        Y_{n, \varepsilon_j} \rightarrow Y_{n} \quad \text { in } C(\overline{Q_n}).
    \end{equation}
    The lower bound in \eqref{lemma_Yn_prop1} follows from \eqref{prop1_varphi}, and \eqref{lemma_Yn_W21p} follows from \eqref{Yne_W21p} and the weak lower semicontinuity of the norm. It remains to verify that $Y_n$ is a strong solution of \eqref{condor_spread_Yn}, since then the upper bound in \eqref{lemma_Yn_prop1} follows from Lemma \ref{lem_comparison} (i), the constant $K_2-K_1$ being a supersolution of \eqref{condor_spread_Yn}. The initial and boundary conditions follow from \eqref{lemma_conv_c}. By \eqref{penalty_bounded} and \eqref{lemma_conv_w}, $\partial_\tau Y_n-\mathcal{L}Y_n\geq0$ almost everywhere in $Q_n$, since the weak limit of nonnegative functions is nonnegative. Let $(x_0,\tau_0)\in Q_n$ with $Y_n(x_0,\tau_0)>\psi_1(x_0)$. By \eqref{lemma_conv_c} and the uniform convergence of $\psi_{1,\varepsilon}$, there are a neighborhood $N$ of $(x_0,\tau_0)$ and $j_0$ such that $Y_{n,\varepsilon_j}-\psi_{1,\varepsilon_j}\geq\varepsilon_j$ in $N$ for $j\geq j_0$, so that the penalty term vanishes in $N$ and $\partial_\tau Y_{n,\varepsilon_j}-\mathcal{L}Y_{n,\varepsilon_j}=0$ in $N$. Passing to the weak limit gives $\partial_\tau Y_n-\mathcal{L}Y_n=0$ almost everywhere in $N$. Hence $\partial_\tau Y_n-\mathcal{L}Y_n=0$ almost everywhere on $\{Y_n>\psi_1\}$, and $Y_n\geq\psi_1$ by \eqref{prop1_varphi}. This completes the proof.
\end{proof}

\begin{theorem} \label{theorem_exist}
    There exists a unique bounded strong solution $Y\in\mathcal{W}_p\left((-\infty,\ln k_1)\times(0,T]\right)$ of the problem \eqref{condor_spread_left}. Furthermore, we have
    \begin{equation} \label{condor_spread_prop1}
    \psi_1(x) \leq Y(x,\tau) \leq K_{2}-K_{1}, \qquad (x,\tau)\in(-\infty,\ln k_1]\times[0,T],
    \end{equation}
    \begin{equation} \label{condor_spread_prop2}
        \partial_{\tau}Y \geq 0, \quad 0 \leq \partial_{x}Y\leq e^x \qquad \text{ almost everywhere in } (-\infty,\ln k_1)\times(0,T].
    \end{equation}
\end{theorem}
\begin{proof}
    The uniqueness follows from Lemma \ref{lem_comparison}. Let $R>\ln\frac{2}{K_1}$ and $n>R$, and let $Y_n$ be the solution given by Lemma \ref{lemma_Yn}. Since $Y_n\in W^{2,1}_{p,\rm{loc}}(Q_n)$, its derivatives coincide almost everywhere on the coincidence set $\{Y_n=\psi_1\}$ with those of $\psi_1$, namely $\partial_\tau Y_n=0$ and $\mathcal{L}Y_n=\mathcal{L}\psi_1$ almost everywhere on $\{Y_n=\psi_1\}\cap\{x\neq\ln K_1\}$, while $\partial_\tau Y_n-\mathcal{L}Y_n=0$ almost everywhere on $\{Y_n>\psi_1\}$. Hence the problem \eqref{condor_spread_Yn} is equivalent to the linear problem
    \begin{equation} \label{Yn_linear}
        \left\{\begin{array}{ll}
        \partial_{\tau} Y_{n}-\mathcal{L} Y_{n}=f_n(x, \tau), & (x,\tau) \in Q_n,\\[0.3cm]
        Y_{n}(-n, \tau)=K_2-k_1, \quad Y_{n}(\ln k_{1}, \tau)=K_2-K_1, & \tau \in[0, T], \\[0.3cm]
        Y_{n}(x,0)=\psi_1(x), & x \in[-n, \ln k_{1}],
        \end{array}\right.
    \end{equation}
    where
    \begin{equation} \label{def_fn}
        \begin{aligned}
        f_n(x,\tau)&=-\chi_{\left\{Y_{n}=\psi_1\right\}}(x,\tau)\,\mathcal{L}\psi_1(x) \\[0.2cm]
        &=\chi_{\left\{Y_{n}=\psi_1\right\}}(x,\tau)\left[\chi_{\{x<\ln K_1\}}\,r(K_{2}-k_{1})+\chi_{\{x>\ln K_1\}}\left(qe^{x} - r(K_{1}+k_1-K_{2})\right)\right]
        \end{aligned}
    \end{equation}
    almost everywhere in $Q_n$, and $ \chi_{A} $ denotes the characteristic function of the set $ A $. In particular the inequality in the second line of \eqref{condor_spread_Yn} becomes the equality $\partial_\tau Y_n-\mathcal{L}Y_n=-\mathcal{L}\psi_1$ on the coincidence set. Then we see that
    \begin{equation*}
         |f_n(x, \tau)| \leq qk_{1} + r(K_{2}-K_{1}+k_{1}) \quad \text { for almost all } (x, \tau) \in Q_n,
    \end{equation*}
    and this bound is independent of $ n $. The data of \eqref{Yn_linear} are independent of $n$ on $[-R,\ln k_1]$, and $\psi_1$ is smooth on $[-R,\ln k_1]\setminus(\ln K_1-\rho,\ln K_1+\rho)$. Therefore the local $W^{2,1}_p$ estimates, see \cite[Chapter IV, Section 10]{ladyzhenskaia1988linear} and \cite[Theorem 7.22]{lieberman1996second}, and \eqref{lemma_Yn_prop1} give
    \begin{equation} \label{Yn_uniform_estimate}
        \left\|Y_{n}\right\|_{W_{p}^{2,1}\left(((-R, \ln k_{1}) \times (0, T]) \setminus B_{\rho}(\ln K_{1},0) \right)}
        \leq C(R,\rho,p)
    \end{equation}
    for some constant $ C(R,\rho,p) $ independent of $ n $. As in the proof of Lemma \ref{lemma_Yn}, the H\"{o}lder estimate up to the boundary gives a bound for $Y_n$ in $C^{\alpha,\alpha/2}([-R,\ln k_1]\times[0,T])$ independent of $n$. Letting $ n  \to \infty $ along a subsequence and then $R\to\infty$ by a diagonal argument, we obtain a function $Y\in\mathcal{W}_p((-\infty,\ln k_1)\times(0,T])$ with $Y_n\rightharpoonup Y$ weakly in $W^{2,1}_p$ on each set $((-R, \ln k_{1}) \times (0, T]) \setminus B_{\rho}(\ln K_{1},0)$ and $Y_n\to Y$ locally uniformly on $(-\infty,\ln k_1]\times[0,T]$. Exactly as at the end of the proof of Lemma \ref{lemma_Yn}, $Y$ is a strong solution of \eqref{condor_spread_left}, and \eqref{condor_spread_prop1} follows from \eqref{lemma_Yn_prop1}.

    We now prove \eqref{condor_spread_prop2} for the solution $Y$ by comparison. Let $\delta\in(0,T)$ and $w(x,\tau):=Y(x,\tau+\delta)$ on $(-\infty,\ln k_1)\times(0,T-\delta]$. Then $w\geq\psi_1$, $\partial_\tau w-\mathcal{L}w\geq0$ almost everywhere and $w(\ln k_1,\tau)=K_2-K_1$, so $w$ is a supersolution of \eqref{condor_spread_left} on $(-\infty,\ln k_1)\times(0,T-\delta]$, and Lemma \ref{lem_comparison} (i) gives $Y(x,\tau+\delta)\geq Y(x,\tau)$. Hence $\partial_\tau Y\geq0$.
    Next let $\delta>0$ and $w(x,\tau):=Y(x-\delta,\tau)$ on $(-\infty,\ln k_1)\times(0,T]$. Since $\psi_1$ is nondecreasing, on the set $\{w>Y\}$ we have $w>Y\geq\psi_1(x)\geq\psi_1(x-\delta)$, so $\partial_\tau w-\mathcal{L}w=0$ almost everywhere there, while $\partial_\tau Y-\mathcal{L}Y\geq0$. Moreover $w(x,0)=\psi_1(x-\delta)\leq\psi_1(x)$ and $w(\ln k_1,\tau)\leq K_2-K_1=Y(\ln k_1,\tau)$ by \eqref{condor_spread_prop1}. Lemma \ref{lem_max_principle} applied to $w-Y$ gives $Y(x-\delta,\tau)\leq Y(x,\tau)$, that is, $\partial_xY\geq0$.
    Finally let $W(x,\tau):=Y(x,\tau)-e^x$ and $\widetilde{W}(x,\tau):=Y(x-\delta,\tau)-e^{x-\delta}$. Since $\psi_1(x)-e^x=K_2-k_1-\min\{e^x,K_1\}$ is nonincreasing in $x$, we have
    \begin{equation*}
        W-\widetilde{W}=\left[Y(x,\tau)-\psi_1(x)\right]-\left[Y(x-\delta,\tau)-\psi_1(x-\delta)\right]+\left[\psi_1(x)-e^x\right]-\left[\psi_1(x-\delta)-e^{x-\delta}\right],
    \end{equation*}
    where the last difference is nonpositive. Hence on the set $\{W>\widetilde{W}\}$ we have $Y(x,\tau)-\psi_1(x)>Y(x-\delta,\tau)-\psi_1(x-\delta)\geq0$, so that $\partial_\tau Y-\mathcal{L}Y=0$ almost everywhere there, and using $\mathcal{L}e^x=-qe^x$ we obtain
    \begin{equation*}
        (\partial_\tau-\mathcal{L})(W-\widetilde{W})=-qe^x-\left[(\partial_\tau-\mathcal{L})Y(x-\delta,\tau)-qe^{x-\delta}\right]\leq -q\left(e^x-e^{x-\delta}\right)\leq0
    \end{equation*}
    almost everywhere on $\{W>\widetilde{W}\}$. At $\tau=0$ we have $W-\widetilde{W}=\left[\psi_1(x)-e^x\right]-\left[\psi_1(x-\delta)-e^{x-\delta}\right]\leq0$, and at $x=\ln k_1$ we have, by \eqref{condor_spread_prop1},
    \begin{equation*}
        \widetilde{W}(\ln k_1,\tau)\geq\psi_1(\ln k_1-\delta)-e^{\ln k_1-\delta}=K_2-k_1-\min\{e^{\ln k_1-\delta},K_1\}\geq K_2-k_1-K_1=W(\ln k_1,\tau).
    \end{equation*}
    Since $W-\widetilde{W}$ is bounded, Lemma \ref{lem_max_principle} gives $W\leq\widetilde{W}$, that is, $Y(x,\tau)-Y(x-\delta,\tau)\leq e^x-e^{x-\delta}$. Dividing by $\delta$ and letting $\delta\searrow0$ yields $\partial_xY\leq e^x$.
\end{proof}

Similar to the problem \eqref{condor_spread_left}, we obtain the existence and uniqueness for the problem \eqref{condor_spread_right}. Here the truncation is at $x=n$ with the boundary value $\psi_2(n)=k_2-K_1$, and $n>\ln(2K_2)$.

\begin{theorem} \label{thm2.4}
    There exists a unique bounded strong solution $Y\in\mathcal{W}_p\left((\ln k_2,\infty)\times(0,T]\right)$ of the problem \eqref{condor_spread_right}. Furthermore, we have
    \begin{equation} \label{condor_spread_prop3}
        \psi_2(x) \leq Y(x,\tau) \leq K_2-K_1, \qquad
        \partial_{\tau}Y \geq 0, \qquad -e^x \leq \partial_x Y \leq 0
    \end{equation}
    on $[\ln k_2,\infty)\times[0,T]$, the last two almost everywhere.
\end{theorem}

Combining the two half-line constructions gives a global price with the regularity appropriate to the payoff.
\begin{theorem} \label{thm_full_existence}
    The problem \eqref{condor_spread_Y} has a unique piecewise strong solution $Y$. Its two half-line restrictions belong to $\mathcal W_p$ for every $1<p<\infty$. It satisfies $Y_0\leq Y\leq K_2-K_1$ and equals $K_2-K_1$ on $[\ln k_1,\ln k_2]\times[0,T]$. Moreover $(\partial_\tau-\mathcal L)Y\geq0$ in the sense of distributions on $\mathbb R\times(0,T)$, with equality in the continuation region. Consequently $V(s,t):=Y(\ln s,T-t)$ is the value function \eqref{optimal_stopping}.
\end{theorem}
\begin{proof}
    Let $Y^-$ and $Y^+$ be the bounded strong solutions given by Theorems \ref{theorem_exist} and \ref{thm2.4}. Define $Y=Y^-$ for $x<\ln k_1$, $Y=K_2-K_1$ for $\ln k_1\leq x\leq\ln k_2$, and $Y=Y^+$ for $x>\ln k_2$. The boundary values make this function continuous. It is a piecewise strong solution and satisfies the stated bounds. Uniqueness follows by applying Lemma \ref{lem_comparison} separately on the two half lines. The constructions for different $p$ give the same function by this uniqueness, so both restrictions belong to every stated $\mathcal W_p$ class.

    To verify the distributional assertion, write $\ell_i=\ln k_i$ and take $p>3$. The one-sided spatial derivatives exist for positive times, and the derivative estimates \eqref{condor_spread_prop2} and \eqref{condor_spread_prop3} give
    \begin{align*}
        J_1(\tau)&:=Y_x(\ell_1+,\tau)-Y_x(\ell_1-,\tau)=-Y_x^-(\ell_1,\tau)\leq0,\\
        J_2(\tau)&:=Y_x(\ell_2+,\tau)-Y_x(\ell_2-,\tau)=Y_x^+(\ell_2,\tau)\leq0.
    \end{align*}
    Integrating by parts on the three intervals shows that the singular part of $(\partial_\tau-\mathcal L)Y$ is
    \begin{equation*}
        -\frac{\sigma^2}{2}\bigl(J_1(\tau)\delta_{\ell_1}+J_2(\tau)\delta_{\ell_2}\bigr),
    \end{equation*}
    a nonnegative measure, where $\delta_{\ell_i}$ is the Dirac measure in the spatial variable. The regular part is nonnegative on the two half lines and equals $r(K_2-K_1)>0$ on the plateau. All singular terms are supported on the contact set, and the equation holds on the continuation set. This also explains why global $W^{2,1}_p$ regularity across the inner strikes is not asserted. Finally apply Theorem \ref{thm_verification}.
\end{proof}

We shall also need the uncapped version of \eqref{condor_spread_left}, in which the obstacle $\psi_1$ is considered on the whole real line without the condition at $x=\ln k_1$. This is the problem of an American call option with strike $K_1$ plus the cash amount $K_2-k_1$, namely
\begin{equation} \label{uncapped_left}
   \begin{cases}
    \; \partial_{\tau} Y-\mathcal{L} Y = 0, \quad
    \text{ if } Y > \psi_1(x),
    & (x, \tau) \in \mathbb{R} \times (0,T] ,\\[0.3cm]
    \; \partial_{\tau} Y-\mathcal{L} Y \geq 0, \quad \text{ if } Y =\psi_1(x), & (x,\tau) \in \mathbb{R} \times (0,T] ,\\[0.3cm]
    \; Y(x, 0)=\psi_1(x),& x \in \mathbb{R}.
   \end{cases}
\end{equation}
\begin{proposition} \label{prop_uncapped}
    The problem \eqref{uncapped_left} has a unique strong solution $Y^{\infty}\in\mathcal{W}_p(\mathbb{R}\times(0,T])$ in the class of functions with $|Y^\infty|\leq Ce^{|x|}$. It satisfies
    \begin{equation} \label{uncapped_bounds}
        \psi_1(x) \leq Y^{\infty}(x,\tau) \leq e^x+K_{2}-k_{1}, \qquad \partial_{\tau}Y^{\infty} \geq 0, \qquad 0 \leq \partial_{x}Y^{\infty}\leq e^x,
    \end{equation}
    the last two almost everywhere in $\mathbb{R}\times(0,T]$.
\end{proposition}
\begin{proof}
    For all sufficiently large $n$, solve the truncated problem on $(-n,n)\times(0,T]$ with boundary values $\psi_1(\pm n)$ by the penalty construction of Lemma \ref{lemma_Yn}. In that construction replace $C_0$ by
    \begin{equation*}
        C_0(n):=(r+q)e^n+r(1+K_2),
    \end{equation*}
    and choose the smoothing parameter small enough that $\psi_{1,\varepsilon}=\psi_1$ at both endpoints. The calculation following \eqref{prop1_varphi} then makes $\psi_{1,\varepsilon}$ a subsolution on this fixed interval. The functions $e^x+K_2-k_1$ and $e^x+K_2-k_1+C_0(n)\tau$ replace the bounded supersolutions used there. Indeed,
    \begin{equation*}
        (\partial_\tau-\mathcal L)(e^x+K_2-k_1)
        =qe^x+r(K_2-k_1)\geq0,
    \end{equation*}
    and the additional term $C_0(n)\tau$ dominates the penalty, whose argument is nonnegative above $\psi_{1,\varepsilon}$.

    First let $\varepsilon\searrow0$ with $n$ fixed. This gives a strong solution $Y_n^\infty$ of the truncated obstacle problem. Comparison with the unpenalized supersolution $e^x+K_2-k_1$ gives
    \begin{equation*}
        \psi_1\leq Y_n^\infty\leq e^x+K_2-k_1.
    \end{equation*}
    The constants $C_0(n)$ need not be uniform in $n$ and are not used in the subsequent passage to the whole line. Instead, as in \eqref{def_fn},
    \begin{equation*}
        (\partial_\tau-\mathcal L)Y_n^\infty
        =-\mathbf1_{\{Y_n^\infty=\psi_1\}}\mathcal L\psi_1
        \quad\hbox{almost everywhere}.
    \end{equation*}
    Writing $D:=K_2-k_1$ and $c:=K_1+k_1-K_2$, the right-hand side is bounded in absolute value on $[-R,R]\times(0,T]$ by
    \begin{equation*}
        rD+qe^R+r|c|,
    \end{equation*}
    independently of $n>R+1$. The local H\"older and $W^{2,1}_p$ estimates used in Theorem \ref{theorem_exist}, with a neighborhood of $(\ln K_1,0)$ removed for the latter estimates, are therefore uniform in $n$ on each fixed bounded set. A diagonal limit gives a strong solution $Y^\infty\in\mathcal W_p(\mathbb R\times(0,T])$ with the asserted value bounds. On compact subsets where the limit is strictly above $\psi_1$, local uniform convergence makes the approximating equations homogeneous for all sufficiently large $n$, which verifies the equation in continuation. Uniqueness follows from Lemma \ref{lem_comparison}.

    Comparing $Y^\infty(x,\tau+\delta)$ with $Y^\infty(x,\tau)$ proves $\partial_\tau Y^\infty\geq0$. Comparing $Y^\infty(x+\delta,\tau)$ with $Y^\infty(x,\tau)$, whose obstacles are ordered, proves $\partial_xY^\infty\geq0$. For the upper derivative bound set $W:=Y^\infty-e^x$ and $W_\delta(x,\tau):=W(x+\delta,\tau)$, where $\delta>0$. The obstacle $\psi_1-e^x$ is nonincreasing. Thus on $\{W_\delta>W\}$ the shifted solution is in continuation and
    \begin{equation*}
        (\partial_\tau-\mathcal L)(W_\delta-W)
        \leq-q(e^{x+\delta}-e^x)\leq0.
    \end{equation*}
    The difference has nonpositive initial data and is bounded, because $-K_1\leq W\leq K_2-k_1$. Lemma \ref{lem_max_principle} gives $W_\delta\leq W$. Letting $\delta\searrow0$ proves $\partial_xY^\infty\leq e^x$.
\end{proof}

\section{Properties of the free boundaries} \label{sec4}
    In this section $Y$ denotes the solution of \eqref{condor_spread_left} given by Theorem \ref{theorem_exist}, unless the problem \eqref{condor_spread_right} is explicitly mentioned. By Theorem \ref{thm_full_existence}, $Y$ is the restriction of the value function of the condor spread, so the free boundaries studied here are those of the original problem. The exercise region $\mathcal{E}$ and the continuation region $\mathcal{C}$ are those of Definition \ref{def_solution} with $\psi=\psi_1$ and $Q=(-\infty,\ln k_1)\times(0,T]$. By \eqref{condor_spread_prop2}, for each $\tau\in(0,T]$ the function $x\mapsto Y(x,\tau)$ is nondecreasing and the function $x\mapsto Y(x,\tau)-\psi_1(x)$ is nonincreasing on $(\ln K_1,\ln k_1)$. Hence the exercise region consists of a part below $\ln K_1$, where the holder collects the cash amount $K_2-k_1$, and a part above $\ln K_1$, where the call is exercised. We first show that the two parts are separated.

\begin{lemma} \label{lemma_K_1}
        For any $ \tau > 0 $, $ Y(\ln K_1,\tau) > K_2-k_1 $.
    \end{lemma}
    \begin{proof}
        Suppose the assertion of the lemma is false. Then there is a $ \tau^* >0 $ such that  
        \begin{equation*}
            Y(\ln K_1,\tau^*)=K_2-k_1.
        \end{equation*}
        Combining (\ref{condor_spread_prop2}) with $ Y \geq K_2-k_1 $ yields that 
        \begin{equation*}
            Y(x,\tau)=K_2-k_1, \qquad (x,\tau) \in (-\infty,\ln K_1) \times [0,\tau^*].
        \end{equation*}
        From the inequalities (\ref{condor_spread_prop2}) again, we obtain 
        \begin{equation*}
            Y(x,\tau)= e^x-K_1+K_2-k_1, \qquad (x,\tau) \in [\ln K_1,\ln k_1) \times [0,\tau^*].
        \end{equation*}
        Therefore, we have
        \begin{equation*}
            Y(x,\tau)= (e^x-K_1)^{+}+K_2-k_1, \qquad (x,\tau) \in (-\infty,\ln k_1) \times [0,\tau^*],
        \end{equation*}
        but $ (e^x-K_1)^{+}+K_2-k_1 \notin W^{2,1}_{p,\rm{loc}} $, which is a contradiction.  
    \end{proof}

    The following lemma replaces the formal computation of $\partial_\tau Y$ at the initial time, which is not available for a $W^{2,1}_p$ solution, and it will be used repeatedly. Recall that $\mathcal{L}\psi_1=-r(K_2-k_1)$ on $(-\infty,\ln K_1)$ and
    \begin{equation} \label{Lpsi1}
        \mathcal{L}\psi_1(x)=-qe^x+r(K_1+k_1-K_2), \qquad x\in(\ln K_1,\ln k_1).
    \end{equation}
    \begin{lemma} \label{lem_test}
        Let $0\leq\tau_0<\tau_1\leq T$ and let $(x_1,x_2)$ be an interval contained either in $(-\infty,\ln K_1)$ or in $(\ln K_1,\ln k_1)$. If $Y(\cdot,\tau_0)=\psi_1$ on $(x_1,x_2)$ and $Y>\psi_1$ in $(x_1,x_2)\times(\tau_0,\tau_1)$, then $\mathcal{L}\psi_1\geq0$ on $(x_1,x_2)$. In particular, the case $(x_1,x_2)\subset(-\infty,\ln K_1)$ cannot occur. The same statement holds for the solution of \eqref{condor_spread_right} with $\psi_2$ in place of $\psi_1$ and intervals contained in $(\ln k_2,\ln K_2)$ or in $(\ln K_2,\infty)$.
    \end{lemma}
    \begin{proof}
        Let $\eta\in C^\infty_c((x_1,x_2))$ be nonnegative and let $\mathcal{L}^{*}\eta:=\frac{\sigma^2}{2}\eta''-(r-q-\frac{\sigma^2}{2})\eta'-r\eta$ be the formal adjoint of $\mathcal{L}$ applied to $\eta$. Since $\psi_1$ is smooth on $(x_1,x_2)$ and $Y\in W^{2,1}_p$ on $\operatorname{supp}\eta\times(\tau_0,\tau_1)$, we have $(\partial_\tau-\mathcal{L})(Y-\psi_1)=\mathcal{L}\psi_1$ almost everywhere in $(x_1,x_2)\times(\tau_0,\tau_1)$, because $Y>\psi_1$ there. Multiplying by $\eta$, integrating over $(x_1,x_2)\times(\tau_0,\tau)$ and integrating by parts in $x$, we obtain for $\tau\in(\tau_0,\tau_1)$
        \begin{equation*}
            \int_{x_1}^{x_2}\left(Y-\psi_1\right)(x,\tau)\eta(x)\,\mathrm{d}x=\int_{\tau_0}^{\tau}\int_{x_1}^{x_2}\left[\left(Y-\psi_1\right)(x,\tau')\mathcal{L}^{*}\eta(x)+\mathcal{L}\psi_1(x)\eta(x)\right]\mathrm{d}x\,\mathrm{d}\tau',
        \end{equation*}
        where we used $Y(\cdot,\tau_0)=\psi_1$ on $(x_1,x_2)$. The left-hand side is nonnegative. Dividing by $\tau-\tau_0$ and letting $\tau\searrow\tau_0$, the first term on the right-hand side tends to zero, because $Y-\psi_1\to0$ uniformly on $\operatorname{supp}\eta$ as $\tau'\searrow\tau_0$ by the continuity of $Y$. Hence $\int_{x_1}^{x_2}\mathcal{L}\psi_1\,\eta\,\mathrm{d}x\geq0$ for every such $\eta$, that is, $\mathcal{L}\psi_1\geq0$ on $(x_1,x_2)$. Since $\mathcal{L}\psi_1=-r(K_2-k_1)<0$ on $(-\infty,\ln K_1)$, the last assertion follows.
    \end{proof}

    The lower part of the exercise region is never empty.
    \begin{lemma} \label{lemma_nonempty}
        There exists $a_0<\ln K_1$, depending only on $r,q,\sigma$ and the strike prices, such that $Y(x,\tau)=K_2-k_1$ for all $x\leq a_0$ and $\tau\in[0,T]$.
    \end{lemma}
    \begin{proof}
        Let $n_1<0<n_2$ be the roots of the characteristic equation $\frac{\sigma^{2}}{2} n^{2}+\left(r-q-\frac{\sigma^{2}}{2}\right) n-r=0$, so that $\mathcal{L}e^{n_1x}=\mathcal{L}e^{n_2x}=0$. For $a<\ln K_1$ define
        \begin{equation*}
            w_a(x):=
            \begin{cases}
                K_2-k_1, & x\leq a,\\[0.2cm]
                \displaystyle \frac{K_2-k_1}{n_2-n_1}\left[n_2e^{n_1(x-a)}-n_1e^{n_2(x-a)}\right], & x>a.
            \end{cases}
        \end{equation*}
        Then $w_a(a)=K_2-k_1$, $w_a'(a)=0$, so $w_a\in C^1(\mathbb{R})\cap W^{2}_{p,\rm{loc}}(\mathbb{R})$, and $-\mathcal{L}w_a=r(K_2-k_1)>0$ for $x<a$ while $-\mathcal{L}w_a=0$ for $x>a$. Moreover $w_a'(x)=\frac{(K_2-k_1)n_1n_2}{n_2-n_1}\left[e^{n_1(x-a)}-e^{n_2(x-a)}\right]>0$ for $x>a$, so $w_a\geq K_2-k_1=\psi_1$ on $(-\infty,\ln K_1]$, and
        \begin{equation*}
            w_a(x)-(K_2-k_1)\geq\frac{(K_2-k_1)}{n_2-n_1}\left[-n_1e^{n_2(\ln K_1-a)}-(n_2-n_1)\right]\to\infty \quad \text{ as } a\to-\infty
        \end{equation*}
        uniformly for $x\in[\ln K_1,\ln k_1]$. Hence there is $a_0<\ln K_1$ such that $w_{a_0}\geq\psi_1$ on $(-\infty,\ln k_1]$ and $w_{a_0}(\ln k_1)\geq K_2-K_1$. Regarded as a function of $(x,\tau)$, $w_{a_0}$ is a bounded supersolution of \eqref{condor_spread_left}, and Lemma \ref{lem_comparison} (i) gives $Y\leq w_{a_0}$. Since $Y\geq\psi_1=K_2-k_1=w_{a_0}$ on $(-\infty,a_0]$, the lemma follows.
    \end{proof}

    By Lemmas \ref{lemma_K_1} and \ref{lemma_nonempty} and the monotonicity properties \eqref{condor_spread_prop2}, for each $\tau\in(0,T]$ the set $\{x\leq\ln K_1 : Y(x,\tau)=K_2-k_1\}$ is a nonempty closed interval of the form $(-\infty,A(\tau)]$ with $A(\tau)<\ln K_1$, and the set $\{x\in[\ln K_1,\ln k_1] : Y(x,\tau)=e^x-K_1+K_2-k_1\}$ is a nonempty closed interval of the form $[B(\tau),\ln k_1]$ with $B(\tau)>\ln K_1$, since it contains $\ln k_1$ by the boundary condition in \eqref{condor_spread_left}. We therefore define the two free boundaries
    \begin{align*}
        &A(\tau)=\max \left\{x : x \leq \ln K_1, Y(x,\tau)=K_2-k_1 \right\}, \\[0.3cm]
        &B(\tau)=\min \left\{x : \ln K_1 \leq x \leq \ln k_1, Y(x,\tau)=e^x-K_1+K_2-k_1\right\},
    \end{align*}
    so that
    \begin{align*}
        \mathcal{E}&=\left\{(x,\tau)\in Q : x\leq A(\tau)\right\}\cup\left\{(x,\tau)\in Q : B(\tau)\leq x<\ln k_1\right\},\\
        \mathcal{C}&=\left\{(x,\tau)\in Q : A(\tau)<x<B(\tau)\right\}.
    \end{align*}
    When $B(\tau)=\ln k_1$, the upper part of the exercise region inside $Q$ is empty, and contact remains at the prescribed boundary $x=\ln k_1$. Thus upper exercise occurs only upon reaching the plateau. We first analyze the free boundary $A(\tau)$.

    \begin{theorem} \label{a0}
        $A$ is strictly decreasing and continuous on $[0,T]$. Furthermore, we have 
        \begin{equation*}
            A(0)=\lim _{\tau \searrow 0} A(\tau)=\ln K_1.
        \end{equation*}
    \end{theorem}
    \begin{proof}
    Since $Y$ is nondecreasing in $\tau$, its continuation intervals expand with time. Thus $A$ is nonincreasing and the limit $A(0):=\lim_{\tau\searrow0}A(\tau)$ exists. It is finite by Lemma \ref{lemma_nonempty}. If $A(0)<\ln K_1$, any interval compactly contained in $(A(0),\ln K_1)$ is in continuation for all positive times and is in contact at time zero. Lemma \ref{lem_test} would give $\mathcal L\psi_1\geq0$ there, contrary to $\mathcal L\psi_1=-r(K_2-k_1)<0$. Hence $A(0)=\ln K_1$.

    We first establish strict time increase of the solution inside continuation, which will also be useful for the other boundaries. Interior parabolic regularity gives a smooth function $Z:=\partial_\tau Y$ in $\mathcal C\cap\{0<\tau<T\}$, and $Z\geq0$ and $(\partial_\tau-\mathcal L)Z=0$ there. We claim that $Z>0$ there. Suppose that $Z(x_0,\tau_0)=0$ at an interior point. The interval $\mathcal C_{\tau_0}$ contains both $x_0$ and $\ln K_1$, by Lemma \ref{lemma_K_1}. Choose a closed spatial interval, contained in $\mathcal C_{\tau_0}$, with both points in its interior. Continuity of $Y-\psi_1$ puts this entire interval in continuation on a strip $[\tau_0-\varepsilon,\tau_0]$, with $0<\varepsilon<\tau_0$. The strong maximum principle in this rectangle gives $Z(\ln K_1,\tau_0-\varepsilon/2)=0$. Set $\tau_1:=\tau_0-\varepsilon/2$. For every $0<s<\tau_1$, the compact vertical segment $\{\ln K_1\}\times[s,\tau_1]$ lies in $\mathcal C$. Covering it by finitely many overlapping interior rectangles and applying the strong maximum principle successively backwards in time yields $Z(\ln K_1,s)=0$. Thus $Y(\ln K_1,\tau)$ is constant on $(0,\tau_1]$. Its initial continuity makes this constant $\psi_1(\ln K_1)$, contradicting Lemma \ref{lemma_K_1}. This proves the claim using only rectangles contained in continuation.

    Suppose now that $A(\tau_1)=A(\tau_2)=a$ for some $0<\tau_1<\tau_2\leq T$. By monotonicity, $A=a$ throughout this interval. Choose $0<h<\tau_2-\tau_1$ and $0<\eta<\ln K_1-a$, and put
    \begin{equation*}
        H_h(x,\tau):=Y(x,\tau+h)-Y(x,\tau)
        \quad\text{on }(a,a+\eta)\times(\tau_1,\tau_2-h).
    \end{equation*}
    Both arguments lie in continuation, so $(\partial_\tau-\mathcal L)H_h=0$, and the strict positivity just proved gives $H_h>0$ in this rectangle. On its left boundary $H_h(a,\tau)=0$. Spatial $C^1$ fit to the constant obstacle at each of the two times gives $\partial_xH_h(a,\tau)=0$. The parabolic Hopf boundary lemma, applied at an interior time of this flat lateral boundary, instead gives $\partial_xH_h(a,\tau)>0$. This contradiction proves strict decrease. No mixed derivative across a moving boundary has been used. The same argument applies to a flat right endpoint of a continuation interval, with the Hopf derivative negative, provided that endpoint lies strictly inside the spatial domain.

    Finally, monotonicity gives both one-sided limits of $A$ at every positive time. Left continuity, including at $T$, follows directly from continuity of $Y$: if $A(\tau_0-) > A(\tau_0)$, a point strictly between these values is in contact just before $\tau_0$ and hence at $\tau_0$, a contradiction. For $\tau_0<T$, a right jump $A(\tau_0+)<A(\tau_0)$ would give an interval compactly contained between these two values that is in contact at $\tau_0$ and in continuation on a right-hand time strip. Lemma \ref{lem_test} again contradicts $\mathcal L\psi_1=-r(K_2-k_1)<0$. Hence $A$ is continuous on $(0,T]$, and the already established initial limit completes the proof.
\end{proof}

    We now turn to the free boundary $B$. Its behavior depends on the sign of \eqref{Lpsi1}. Let
    \begin{equation} \label{def_xi}
        \xi:=\max\left\{\ln K_1,\ \ln \frac{r(K_1+k_1-K_2)}{q}\right\},
    \end{equation}
    where the second term is understood as $-\infty$ if $q=0$ or $K_1+k_1\leq K_2$. Then $\mathcal{L}\psi_1>0$ on $(\ln K_1,\xi)$ and $\mathcal{L}\psi_1<0$ on $(\xi,\infty)$, except in the case $q=0$ and $K_1+k_1\geq K_2$, in which $\mathcal{L}\psi_1=r(K_1+k_1-K_2)\geq0$ on the whole interval $(\ln K_1,\ln k_1)$.
    \begin{theorem} \label{b0}
    The following statements hold.
    \begin{enumerate}[label=(\roman*)]
        \item If $ q=0 $ and $ K_1+k_1 \geq K_2 $, or if $q>0$ and $\xi\geq\ln k_1$, then $B(\tau)=\ln k_1$ for all $\tau\in(0,T]$. In this case the call is never exercised before the price of the underlying asset reaches $k_1$.
        \item If $q=0$ and $K_1+k_1<K_2$, or if $q>0$ and $\xi<\ln k_1$, then $B$ is nondecreasing and continuous on $(0,T]$, it is strictly increasing on $\{\tau\in(0,T] : B(\tau)<\ln k_1\}$, and
        \begin{equation*}
            B(0):=\lim _{\tau \searrow 0} B(\tau)=\xi .
        \end{equation*}
    \end{enumerate}
    \end{theorem}
    \begin{proof}
    Time monotonicity of $Y$ makes its exercise sets shrink, so $B$ is nondecreasing and $B(0):=\lim_{\tau\searrow0}B(\tau)$ exists. If $B(\tau_0)<\ln k_1$, then $[B(\tau_0),\ln k_1)\times(0,\tau_0]$ lies in the contact set. The strong obstacle inequalities and the derivatives of $Y=\psi_1$ almost everywhere on this set imply
    \begin{equation} \label{b0_observation}
        \mathcal L\psi_1\leq0\qquad\text{on }[B(\tau_0),\ln k_1).
    \end{equation}
    If $q=0$ and $K_1+k_1>K_2$, or if $q>0$ and $\xi\geq\ln k_1$, the expression $\mathcal L\psi_1=-qe^x+r(K_1+k_1-K_2)$ is strictly positive throughout $(\ln K_1,\ln k_1)$, so \eqref{b0_observation} is impossible. Thus $B\equiv\ln k_1$ in these cases.

    In the remaining case of (i), $q=0$ and $K_1+k_1=K_2$, the payoff below the cap is $\max\{s,K_1\}$ and its value at the cap is $k_1$. Fix $s_0\in(K_1,k_1)$ and $t<T$, and set $\theta:=\inf\{u\geq t:S_u\geq k_1\}\wedge T$, with $S_t=s_0$. The verification theorem and optional sampling of the discounted stock give
    \begin{equation*}
        V(s_0,t)\geq\mathbb E[e^{-r(\theta-t)}V_0(S_\theta)]
        =s_0+\mathbb E[e^{-r(T-t)}(K_1-S_T)^+\mathbf1_{\{\theta=T\}}]>s_0,
    \end{equation*}
    because $\mathbb P(\theta=T,\ S_T<K_1)>0$. This proves $B\equiv\ln k_1$ and completes (i).

    In case (ii), $\xi<\ln k_1$. Observation \eqref{b0_observation} implies $B(\tau)\geq\xi$: if $B(\tau)<\xi$, the contact interval would contain points where $\mathcal L\psi_1>0$. If $B(0)>\xi$, an interval compactly contained in $(\xi,B(0))$ would be in continuation at every positive time and in contact initially. Lemma \ref{lem_test} would force $\mathcal L\psi_1\geq0$ there, contrary to its strict negativity for $x>\xi$. Hence $B(0)=\xi$.

    If $B$ were constant on $[\tau_1,\tau_2]$ at a level $b<\ln k_1$, choose $0<h<\tau_2-\tau_1$ and a rectangle $(b-\eta,b)\times(\tau_1,\tau_2-h)$ with $b-\eta>\ln K_1$. The time increment $H_h$ used in the proof of Theorem \ref{a0} is positive inside, solves the homogeneous parabolic equation, and vanishes on $x=b$. Spatial smooth fit gives $\partial_xH_h(b,\tau)=e^b-e^b=0$, whereas the Hopf boundary lemma gives a negative derivative. This proves the asserted strict increase below the cap.

    Left continuity, including at $T$, follows from continuity of $Y$: a point strictly between $B(\tau_0-)$ and $B(\tau_0)$ would be in contact before $\tau_0$ and therefore at $\tau_0$, contradicting the definition of $B(\tau_0)$. If $\tau_0<T$ and $B(\tau_0+)>B(\tau_0)$, choose an interval compactly contained between these two levels. Its points are strictly greater than $\xi$, since $B(\tau_0)\geq\xi$. They are in contact at $\tau_0$ and in continuation immediately afterwards. Lemma \ref{lem_test} contradicts $\mathcal L\psi_1<0$ on this interval. Once $B$ reaches the cap, no right jump is possible. This proves continuity on $(0,T]$.
\end{proof}

We now turn to the problem \eqref{condor_spread_right}. The following lemma is the analogue of Lemma \ref{lemma_K_1}.
\begin{lemma} \label{lemma_K_2}
    For any $ \tau >0,   Y(\ln K_2,\tau) > k_2-K_1 $.
\end{lemma}
The proof is the same as that of Lemma \ref{lemma_K_1}.
By Lemma \ref{lemma_K_2}, by \eqref{condor_spread_prop3} and by the analogue of Lemma \ref{lemma_nonempty} for the problem \eqref{condor_spread_right}, we can define the two free boundaries 
\begin{align*}
    &\widetilde{A}(\tau)=\min \{x : x \geq \ln K_2, Y(x,\tau)=k_2-K_1 \}, \\[0.3cm] 
    &\widetilde{B}(\tau)=\max \{x : \ln k_2\leq x \leq \ln K_2, Y(x,\tau)=K_2-e^x+k_2-K_1\}.
\end{align*}

Similar to Theorem \ref{a0}, we obtain the following theorem. 
\begin{theorem} \label{a0_tilde}
    $\widetilde{A}$ is strictly increasing and continuous on $[0,T]$. Furthermore, we have 
    \begin{equation*}
        \widetilde{A}(0)=\lim _{\tau \searrow 0} \widetilde{A}(\tau)=\ln K_2.
    \end{equation*}
\end{theorem}
\begin{proof}
    By \eqref{condor_spread_prop3}, the continuation intervals of the right-hand problem expand with time, so $\widetilde A$ is nondecreasing. Its initial limit cannot exceed $\ln K_2$: otherwise an interval between $\ln K_2$ and this limit would be in continuation for every positive time, and initially in contact with the constant obstacle $k_2-K_1$. The right-hand version of Lemma \ref{lem_test} would contradict $\mathcal L(k_2-K_1)=-r(k_2-K_1)<0$. Thus $\widetilde A(0)=\ln K_2$.

    The strict positivity argument for $\partial_\tau Y$ in the proof of Theorem \ref{a0} applies with $\ln K_2$ as the fixed point inside all continuation intervals, using Lemma \ref{lemma_K_2}. If $\widetilde A$ were flat at $a$ on a positive time interval, its positive time increment $H_h$ would solve the homogeneous equation in a rectangle immediately to the left of $a$, vanish on $x=a$, and satisfy $\partial_xH_h(a,\tau)=0$ by spatial fit to the constant obstacle. The Hopf boundary lemma gives $\partial_xH_h(a,\tau)<0$, a contradiction. Hence $\widetilde A$ is strictly increasing.

    A left jump, including one at $T$, is excluded by the continuity of $Y$, as in the proof of Theorem \ref{a0}. A right jump at a time smaller than $T$ would release an interval on which the obstacle is the constant $k_2-K_1$, and Lemma \ref{lem_test} excludes it because $\mathcal L(k_2-K_1)=-r(k_2-K_1)<0$. Together with the initial limit this proves continuity on $[0,T]$.
\end{proof}
Similar to Theorem \ref{b0}, we have the following theorem.
\begin{theorem} \label{b0_tilde}
        $\widetilde{B}$ is nonincreasing and continuous on $(0,T]$ with  
        \begin{equation*}
            \widetilde{B}(0)=
            \begin{cases} 
                \ln K_2, & \text{if } q=0, \\[0.3cm]  
                \max\left\{\ln k_2, \min\left\{\ln K_2,\ln \frac{r(K_2+k_2-K_1)}{q}\right\}\right\}, & \text{if } q>0.
            \end{cases}
        \end{equation*}
        Furthermore, $\widetilde{B}$ is strictly decreasing on $\{\tau\in(0,T] : \widetilde{B}(\tau)>\ln k_2\}$, and if $q>0$ and $r(K_2+k_2-K_1)/q\leq k_2$, then $\widetilde{B}(\tau)=\ln k_2$ for all $\tau\in(0,T]$. 
\end{theorem}
\begin{proof}
    Time monotonicity of $Y$ makes $\widetilde B$ nonincreasing, and the boundary value at $\ln k_2$ gives $\widetilde B\geq\ln k_2$. Write $C_2:=K_2+k_2-K_1$, so that
    \begin{equation*}
        \mathcal L\psi_2=qe^x-rC_2\qquad(\ln k_2<x<\ln K_2).
    \end{equation*}
    If $\widetilde B(\tau_0)>\ln k_2$, the contact strip $(\ln k_2,\widetilde B(\tau_0))\times(0,\tau_0]$ and the strong obstacle inequalities force $\mathcal L\psi_2\leq0$ there. Consequently, for $q>0$, $\widetilde B(\tau_0)\leq\ln(rC_2/q)$. If $rC_2/q\leq k_2$, this excludes every such nonempty contact strip, proving $\widetilde B\equiv\ln k_2$.

    In all other cases define
    \begin{equation*}
        \zeta:=
        \begin{cases}
            \ln K_2,&q=0,\\
            \min\{\ln K_2,\ln(rC_2/q)\},&q>0.
        \end{cases}
    \end{equation*}
    Then $\zeta>\ln k_2$, $\widetilde B(\tau)\leq\zeta$, and $\mathcal L\psi_2<0$ on $(\ln k_2,\zeta)$. If $\widetilde B(0):=\lim_{\tau\searrow0}\widetilde B(\tau)<\zeta$, an interval compactly contained in $(\widetilde B(0),\zeta)$ would be in continuation for all positive times and initially in contact, contrary to Lemma \ref{lem_test}. Hence $\widetilde B(0)=\zeta$, giving the stated formula in every case.

    As in the proof of Theorem \ref{a0_tilde}, $\partial_\tau Y>0$ in the interior of continuation. If $\widetilde B$ were constant at a level $b>\ln k_2$ on a time interval, the time-increment argument applies on a rectangle immediately to the right of $b$. On $x=b$, the function $Y$ takes the value $\psi_2(b)$ and has the spatial derivative $-e^b$ at both times $\tau$ and $\tau+h$, so $H_h=\partial_xH_h=0$ there. The Hopf boundary lemma instead gives $\partial_xH_h>0$. This proves strict decrease above the cap.

    Left continuity at every positive time, including $T$, follows from the continuity of $Y$ as before. If there were a right jump downwards at $\tau_0<T$, an interval strictly between $\widetilde B(\tau_0+)$ and $\widetilde B(\tau_0)$ would be in contact at $\tau_0$ and in continuation immediately afterwards. It lies below $\zeta$, so $\mathcal L\psi_2<0$ there, contradicting Lemma \ref{lem_test}. The constant case has no jumps. This proves continuity on $(0,T]$.
\end{proof}

We now relate the free boundaries of the condor spread to those of the uncapped problem \eqref{uncapped_left}. Broadie and Detemple \cite[Theorem 1]{broadie1995american} proved that the exercise boundary of an American capped call option with constant cap $L$ is the minimum of $L$ and the exercise boundary of the uncapped call. The problem \eqref{condor_spread_left} plays the role of the capped option, the plateau level $k_1$ plays the role of the cap, and the problem \eqref{uncapped_left} plays the role of the uncapped option. We first record the properties of the uncapped free boundary. Let $Y^\infty$ be the solution of \eqref{uncapped_left} given by Proposition \ref{prop_uncapped}. By \eqref{uncapped_bounds}, $x\mapsto Y^\infty(x,\tau)$ is nondecreasing and $x\mapsto Y^\infty(x,\tau)-\psi_1(x)$ is nonincreasing on $(\ln K_1,\infty)$, so the set $\{x\geq\ln K_1 : Y^\infty(x,\tau)=\psi_1(x)\}$ is either empty or a closed half line. We define
\begin{equation} \label{def_Binfty}
    B^{\infty}(\tau):=\inf\left\{x\geq\ln K_1 : Y^{\infty}(x,\tau)=\psi_1(x)\right\}\in[\ln K_1,+\infty], \qquad \tau\in(0,T],
\end{equation}
with the convention $\inf\emptyset=+\infty$.
\begin{proposition} \label{prop_uncapped_boundary}
    For every $\tau\in(0,T]$ we have $Y^\infty(\ln K_1,\tau)>K_2-k_1$, so that $B^\infty(\tau)>\ln K_1$, and $B^\infty$ is nondecreasing on $(0,T]$. Moreover, the following statements hold.
    \begin{enumerate}[label=(\roman*)]
        \item If $q=0$ and $K_1+k_1\geq K_2$, then $B^\infty(\tau)=+\infty$ for all $\tau\in(0,T]$.
        \item Otherwise $B^\infty$ is continuous on $(0,T]$ as a map into $(\ln K_1,+\infty]$, it is strictly increasing on $\{\tau : B^\infty(\tau)<\infty\}$, and $\lim_{\tau\searrow0}B^\infty(\tau)=\xi$.
    \end{enumerate}
\end{proposition}
\begin{proof}
    The proof of Lemma \ref{lemma_K_1} applies to $Y^\infty$: the bounds \eqref{uncapped_bounds} and spatial $C^1$ regularity exclude contact at the upward corner $\ln K_1$. Time monotonicity gives that $B^\infty$ is nondecreasing. If it is finite at $\tau_0$, the same contact-strip argument as in \eqref{b0_observation} gives $\mathcal L\psi_1\leq0$ on $[B^\infty(\tau_0),\infty)$. This is impossible when $q=0$ and $K_1+k_1>K_2$.

    When $q=0$ and $K_1+k_1=K_2$, put $X_u=x+(r-\sigma^2/2)u+\sigma W_u$ and
    $u(x,\tau):=\mathbb E[e^{-r\tau}\psi_1(X_\tau)]$.
    This European value is a subsolution of \eqref{uncapped_left}, with the required exponential growth, so $Y^\infty\geq u$ by Lemma \ref{lem_comparison}. Here $\psi_1(x)=\max\{e^x,K_1\}$, and the discounted stock is a martingale. Since the normal distribution of $X_\tau$ assigns positive probability to $X_\tau<\ln K_1$, we have $u(x,\tau)>e^x=\psi_1(x)$ for every $x>\ln K_1$ and $\tau>0$. This proves (i).

    In all other cases, the contact-strip inequality gives $B^\infty(\tau)\geq\xi$. If the initial limit were greater than $\xi$, including the possibility $+\infty$, choose a finite interval strictly between $\xi$ and that limit. It would be in continuation for every positive time and initially in contact. The proof of Lemma \ref{lem_test}, which is local and applies without change to the uncapped problem, would contradict $\mathcal L\psi_1<0$ there. Hence the initial limit equals $\xi$.

    The continuation sections of $Y^\infty$ are intervals containing $\ln K_1$ at all positive times, by \eqref{uncapped_bounds}. Therefore the interior strict time-positivity argument in the proof of Theorem \ref{a0} applies to $Y^\infty$ as well. If $B^\infty$ were constant at a finite level on a positive time interval, the time-increment and Hopf argument in the proof of Theorem \ref{b0} would give the same contradiction, since this boundary is an interior spatial point. Thus $B^\infty$ is strictly increasing whenever finite.

    To check continuity in the extended real line, first suppose that $B^\infty(\tau_0-)$ is strictly smaller than $B^\infty(\tau_0)$, where the latter may be infinite. Choose a finite point strictly between them. It is in contact just before $\tau_0$ and hence at $\tau_0$, a contradiction. This proves left continuity, also at $T$. For a right jump at $\tau_0<T$, the value $B^\infty(\tau_0)$ must be finite. Choose a finite interval strictly between it and $B^\infty(\tau_0+)$, even if the latter is infinite. Its points exceed $\xi$, are in contact at $\tau_0$, and enter continuation immediately afterwards. The uncapped version of Lemma \ref{lem_test} contradicts $\mathcal L\psi_1<0$ there. If $B^\infty(\tau_0)=+\infty$, right continuity is automatic by monotonicity. This proves (ii).
\end{proof}

\begin{theorem} \label{thm_reduction}
    Let $Y$ and $B$ be the solution and the upper free boundary of \eqref{condor_spread_left}, and let $Y^\infty$ and $B^\infty$ be the solution of \eqref{uncapped_left} and its free boundary \eqref{def_Binfty}. Then the following statements hold.
    \begin{enumerate}[label=(\roman*)]
        \item $Y\leq Y^\infty$ on $(-\infty,\ln k_1]\times[0,T]$.
        \item $B(\tau)=\min\left\{B^\infty(\tau),\ \ln k_1\right\}$ for every $\tau\in(0,T]$.
        \item Let $\tau^*:=\sup\{\tau\in(0,T] : B^\infty(\tau)\leq\ln k_1\}$, with $\tau^*:=0$ if this set is empty. Then $Y=Y^\infty$ on $(-\infty,\ln k_1]\times[0,\tau^*]$. In particular, on $[0,\tau^*]$ the lower free boundary $A$ coincides with the lower free boundary
        \begin{equation*}
            A^\infty(\tau):=\max\left\{x\leq\ln K_1 : Y^\infty(x,\tau)=K_2-k_1\right\}
        \end{equation*}
        of the uncapped problem, and $A^\infty(\tau)\leq A(\tau)$ for all $\tau\in(0,T]$.
    \end{enumerate}
\end{theorem}
\begin{proof}
    The restriction of $Y^\infty$ to $(-\infty,\ln k_1]\times[0,T]$ is bounded by $K_2$, it satisfies $Y^\infty\geq\psi_1$ and $\partial_\tau Y^\infty-\mathcal{L}Y^\infty\geq0$ almost everywhere, and $Y^\infty(\ln k_1,\tau)\geq\psi_1(\ln k_1)=K_2-K_1$. Hence it is a supersolution of \eqref{condor_spread_left}, and Lemma \ref{lem_comparison} (i) gives (i). In particular every point of the exercise region of $Y^\infty$ with $x\leq\ln k_1$ belongs to the exercise region of $Y$. This gives $B(\tau)\leq B^\infty(\tau)$ whenever $B^\infty(\tau)\leq\ln k_1$, and since $B(\tau)\leq\ln k_1$ always, we obtain $B(\tau)\leq\min\{B^\infty(\tau),\ln k_1\}$. The lower coincidence set of $Y^\infty$ is nonempty. Indeed, the positive characteristic root in Lemma \ref{lemma_nonempty} satisfies $n_2\geq1$, since $q\geq0$. The positive $e^{n_2(x-a)}$ term in $w_a$ therefore dominates $e^x+K_2-k_1$ on $[\ln K_1,\infty)$ when $a$ is sufficiently negative, while $w_a\geq K_2-k_1$ everywhere. Thus $w_a\geq\psi_1$ on $\mathbb R$ and $(\partial_\tau-\mathcal L)w_a\geq0$. Although $w_a$ may grow faster than $e^x$, the difference $Y^\infty-w_a$ satisfies the upper growth bound in Lemma \ref{lem_max_principle}, because $w_a\geq0$. On its positive set $Y^\infty$ is in continuation, so this difference satisfies the required differential inequality and is nonpositive initially. That lemma gives $Y^\infty\leq w_a$. Hence $Y^\infty=K_2-k_1$ on $(-\infty,a]\times[0,T]$. The lower coincidence set is consequently a nonempty closed half line, so $A^\infty$ is well defined, and (i) gives $A^\infty(\tau)\leq A(\tau)$.

    For the reverse inequality in (ii) we distinguish three cases. If $q=0$ and $K_1+k_1\geq K_2$, then $B^\infty\equiv+\infty$ by Proposition \ref{prop_uncapped_boundary} (i) and $B\equiv\ln k_1$ by Theorem \ref{b0} (i), so (ii) holds and $\tau^*=0$. If $q>0$ and $\xi\geq\ln k_1$, then $B\equiv\ln k_1$ by Theorem \ref{b0} (i), while $B^\infty(\tau)\geq\lim_{u\searrow0}B^\infty(u)=\xi\geq\ln k_1$ by Proposition \ref{prop_uncapped_boundary} (ii), so (ii) holds again with $\tau^*=0$, because $B^\infty$ is strictly increasing. In the remaining case $\xi<\ln k_1$, Proposition \ref{prop_uncapped_boundary} (ii) gives $\tau^*>0$ and, by continuity and monotonicity, $B^\infty(\tau)\leq\ln k_1$ for $\tau\in(0,\tau^*]$ and $B^\infty(\tau)>\ln k_1=B^\infty(\tau^*)$ for $\tau\in(\tau^*,T]$ if $\tau^*<T$. For $\tau\in(0,\tau^*]$ we have $\ln k_1\geq B^\infty(\tau)$ and therefore $Y^\infty(\ln k_1,\tau)=\psi_1(\ln k_1)=K_2-K_1$. Consequently the restriction of $Y^\infty$ to $(-\infty,\ln k_1]\times[0,\tau^*]$ is a bounded strong solution of \eqref{condor_spread_left} with $T$ replaced by $\tau^*$, and the uniqueness in Lemma \ref{lem_comparison} gives $Y=Y^\infty$ on $(-\infty,\ln k_1]\times[0,\tau^*]$, which is (iii). For $\tau\in(0,\tau^*]$ the sets $\{x\in[\ln K_1,\ln k_1] : Y(x,\tau)=\psi_1(x)\}$ and $\{x\in[\ln K_1,\ln k_1] : Y^\infty(x,\tau)=\psi_1(x)\}$ coincide, and the second one is $[B^\infty(\tau),\ln k_1]$, so $B(\tau)=B^\infty(\tau)=\min\{B^\infty(\tau),\ln k_1\}$. Finally let $\tau\in(\tau^*,T]$ and $x\in[\ln K_1,\ln k_1)$. Since $\partial_\tau Y\geq0$ and $x<\ln k_1=B^\infty(\tau^*)$, we get
    \begin{equation*}
        Y(x,\tau)\geq Y(x,\tau^*)=Y^\infty(x,\tau^*)>\psi_1(x).
    \end{equation*}
    Hence $B(\tau)=\ln k_1=\min\{B^\infty(\tau),\ln k_1\}$, which completes the proof of (ii).
\end{proof}

\begin{remark} \label{rem_reduction}
    Theorem \ref{thm_reduction} extends the constant-cap boundary reduction of Broadie and Detemple \cite[Theorem 1]{broadie1995american} to the payoff with a cash component. It identifies the inner boundary $B$ as the uncapped boundary cut off at $\ln k_1$. The value functions coincide on the initial time interval $[0,\tau^*]$ defined in that theorem. If $0<\tau^*<T$, the uncapped boundary first reaches $\ln k_1$ at $\tau^*$. The definition also covers the cases in which the cap is active immediately or is not reached within the horizon. For $\tau>\tau^*$, $Y\leq Y^\infty$, with strict inequality at $x=\ln k_1$. Strict inequality need not hold everywhere: both functions equal the cash payoff for sufficiently negative $x$. The outer boundary satisfies $A\geq A^\infty$, and no cutoff formula for $A$ is asserted after $\tau^*$.

    The same argument applies on the put side. Construct the uncapped solution $\widetilde Y^\infty$ on $\mathbb R$ with obstacle $\psi_2=(K_2-e^x)^++k_2-K_1$, and put
    \begin{equation*}
        \widetilde B^\infty(\tau):=\sup\{x\leq\ln K_2:
        \widetilde Y^\infty(x,\tau)=\psi_2(x)\},
        \qquad \sup\emptyset:=-\infty.
    \end{equation*}
    This boundary is nonincreasing. Its initial limit is $\ln K_2$ if $q=0$ and $\min\{\ln K_2,\ln[r(K_2+k_2-K_1)/q]\}$ if $q>0$. The contact-strip, strict time-positivity and jump-exclusion arguments above apply with left and right interchanged. Let $\widetilde\tau^*:=\sup\{\tau\in(0,T]:\widetilde B^\infty(\tau)\geq\ln k_2\}$, with $\widetilde\tau^*:=0$ if this set is empty. Before this time the two problems have the same data at $\ln k_2$, so uniqueness identifies their solutions on $[\ln k_2,\infty)\times[0,\widetilde\tau^*]$. If $\widetilde\tau^*=0$, the initial limit of $\widetilde B^\infty$ is at most $\ln k_2$, and Theorem \ref{b0_tilde} gives $\widetilde B\equiv\ln k_2$. If $0<\widetilde\tau^*<T$, time monotonicity after $\widetilde\tau^*$ keeps every point of $(\ln k_2,\ln K_2]$ in continuation. Consequently
    \begin{equation*}
        \widetilde B(\tau)=\max\{\widetilde B^\infty(\tau),\ln k_2\},
        \qquad \tau\in(0,T].
    \end{equation*}
    The outer cash-exercise boundaries $A$ and $\widetilde A$ always exist and are strictly monotone by Theorems \ref{a0} and \ref{a0_tilde}. These additional boundaries arise from the positive cash components of the payoffs.
\end{remark}

\medskip
We close this section with a comparative static analysis with respect to the volatility $\sigma$. The cap requires a restriction on the parameters: convexity below the cap need not hold when $q=0$ and $K_1+k_1>K_2$. The following sufficient condition includes the case $K_1+k_1=K_2$ for every $q\geq0$.
\begin{lemma} \label{lem3.1}
    Let $Y$ be the solution of \eqref{condor_spread_left}, and suppose that
    \begin{equation} \label{cond_convex_cap}
        qk_1\geq r(K_1+k_1-K_2).
    \end{equation}
    Then $s\mapsto Y(\ln s,\tau)$ is convex on $(0,k_1)$ for every $\tau\in[0,T]$. Equivalently,
    \begin{equation*}
        \partial_{xx}Y-\partial_xY\geq0
        \qquad\text{almost everywhere in }(-\infty,\ln k_1)\times(0,T].
    \end{equation*}
\end{lemma}
\begin{proof}
    Write
    \begin{equation*}
        L:=k_1,\quad D:=K_2-k_1,\quad c:=K_1-D,
        \quad M:=L-c=K_2-K_1,\quad
        \kappa:=\frac{\sigma^2}{2},\quad \mu:=r-q,
    \end{equation*}
    and set $\phi(s):=(s-K_1)^++D$ and
    $\mathcal G v:=\kappa s^2v''+\mu sv'-rv$.
    We use implicit time discretization, so that no time derivative of the solution at a moving free boundary is needed.

    Let $w(s):=w_a(\ln s)$ be the barrier in Lemma \ref{lemma_nonempty}, with $a$ chosen so that $w\geq\phi$ on $(0,L]$ and $w(L)\geq M$. Set $s_0:=e^a<K_1$, choose $0<\ell<s_0$, and work first on $[\ell,L]$. Then $w=D$ on $[\ell,s_0]$ and $-\mathcal Gw\geq0$. For $h=T/N$, let $U_0=\phi$ and define $U_j$, $1\leq j\leq N$, by
    \begin{equation} \label{convex_implicit}
        \min\left\{U_j-\phi,\,
        \frac{U_j-U_{j-1}}h-\mathcal G U_j\right\}=0,
        \qquad U_j(\ell)=D,\quad U_j(L)=M.
    \end{equation}
    These elliptic obstacle problems have unique variational solutions. Indeed, with
    \begin{equation*}
        \rho(s):=s^{\mu/\kappa-2},\qquad
        \mathcal K:=\{z\in H^1(\ell,L):z\geq\phi,\ z(\ell)=D,\ z(L)=M\},
    \end{equation*}
    $U_j$ is the unique minimizer over $\mathcal K$ of
    \begin{equation*}
        \frac1{2h}\int_\ell^L\rho\,|z-U_{j-1}|^2\,\mathrm ds
        +\mathcal E(z),\qquad
        \mathcal E(z):=\frac12\int_\ell^L
        \rho\left(\kappa s^2|z'|^2+rz^2\right)\,\mathrm ds.
    \end{equation*}
    The identity $(\kappa s^2\rho)'=\mu s\rho$ gives \eqref{convex_implicit}. We recall the elementary one-dimensional regularity used below. The variational inequality implies
    \begin{equation*}
        (\kappa s^2\rho U_j')'
        =\rho\bigl[(r+h^{-1})U_j-h^{-1}U_{j-1}\bigr]-\nu_j,
        \qquad \nu_j\geq0,
    \end{equation*}
    where $\nu_j$ is a measure supported on the contact set. Thus the derivative has one-sided traces and cannot jump upward. At a contact point $s\neq K_1$, the inequality $U_j\geq\phi$ requires $U_j'(s-)\leq\phi'(s)\leq U_j'(s+)$, so both traces equal $\phi'(s)$. Contact at $K_1$ is impossible, because it would require the left derivative to be at most $0$ and the right derivative to be at least $1$. On the open continuation set the linear equation applies. These observations give $U_j\in C^1([\ell,L])$ and $U_j\in C^2$ on continuation, with one-sided second-derivative limits at its endpoints, obtained from the linear equation and the continuity of $U_{j-1}$.

    Elliptic comparison, by induction in $j$, gives
    \begin{equation} \label{convex_discrete_bounds}
        \phi\leq U_{j-1}\leq U_j\leq M,
        \qquad U_j\leq w.
    \end{equation}
    For the time monotonicity, compare two consecutive problems, whose right-hand sides are $U_{j-1}/h$ and $U_{j-2}/h$, and note that the first step follows from $U_1\geq\phi$. The constant $M$ is a supersolution, and $w$ is a supersolution whenever $U_{j-1}\leq w$. In particular, $U_j=D$ on $[\ell,s_0]$. Extend $U_j$ by $D$ to $(0,\ell)$. This extension still satisfies \eqref{convex_implicit} on $(0,L)$.

    We also have
    \begin{equation} \label{convex_discrete_delta}
        0\leq U_j'\leq1.
    \end{equation}
    Indeed, on each continuation interval $P_j:=U_j'$ solves
    \begin{equation*}
        -\kappa s^2P_j''-(2\kappa+\mu)sP_j'
        +(h^{-1}+q)P_j=h^{-1}U_{j-1}'
    \end{equation*}
    in the weak sense. At a contact endpoint, smooth fit gives $P_j=0$ or $P_j=1$. If the interval reaches $L$, the bounds $\phi\leq U_j\leq M$ and $U_j(L)=\phi(L)=M$ give $0\leq P_j(L-)\leq1$. There is no continuation interval at $\ell$. Starting from $0\leq U_0'\leq1$ almost everywhere, comparison with the constants $0$ and $1$ in this linear equation proves \eqref{convex_discrete_delta} by induction. In particular, $U_j-D$ is nondecreasing and $U_j-(s-c)$ is nonincreasing. The lower and upper coincidence sets are therefore intervals, and the continuation set is one interval containing $K_1$.

    We now prove convexity by induction. The distributional second derivative of $U_0$ is the positive measure $\delta_{K_1}$. Suppose that $U_{j-1}$ is convex and write $G_j:=U_j''$ in the continuation set. Differentiating its linear equation twice in the sense of distributions gives
    \begin{equation} \label{convex_discrete_gamma}
        \kappa s^2G_j''+(4\kappa+\mu)sG_j'
        +(2\kappa+2\mu-r-h^{-1})G_j
        =-h^{-1}U_{j-1}''\leq0.
    \end{equation}
    Choose $N$ large enough that the zeroth-order coefficient is negative. At an endpoint in $(\ell,L)$ of a continuation interval, value matching and $C^1$ fit to the affine obstacle give $G_j\geq0$ as a one-sided limit: this follows directly by Taylor expansion of the nonnegative function $U_j-\phi$. The same observation applies at $L$ if continuation does not reach $L$. If continuation does reach $L$, the equation, $U_j(L)=U_{j-1}(L)=M$, and \eqref{convex_discrete_delta} give
    \begin{equation*}
        \kappa L^2G_j(L-)=rM-\mu L U_j'(L-)
        \geq
        \begin{cases}
            rM-(r-q)L=qL-rc\geq0,&q\leq r,\\
            rM>0,&q>r.
        \end{cases}
    \end{equation*}
    There is no continuation interval at $\ell$, because $U_j=D$ on $[\ell,s_0]$. The formula
    \begin{equation*}
        G_j=\frac{(r+h^{-1})U_j-\mu sU_j'-h^{-1}U_{j-1}}{\kappa s^2}
    \end{equation*}
    shows first that $G_j$ is bounded up to the endpoints, and then that $G_j\in W^{1,\infty}$ on the continuation interval, since $U_{j-1}$ is Lipschitz and $U_j'$ consequently is Lipschitz there. The weak maximum principle applied to \eqref{convex_discrete_gamma} therefore gives $G_j\geq0$. This principle also covers the first step, where the right-hand side is a negative measure: after multiplication by a positive integrating factor, test the distributional inequality against the negative part of $G_j$, which belongs to $H^1_0$ of the continuation interval. The boundary terms vanish and the gradient and zeroth-order terms are nonnegative. Outside continuation the obstacle is affine, and $C^1$ fit creates no atom in $U_j''$. Hence $U_j$ is convex on $(0,L)$.

    For completeness, convergence of these approximations follows from the variational formulation without any assertion about regularity of the moving boundaries. Minimality with competitor $U_{j-1}$ gives
    \begin{equation*}
        \mathcal E(U_j)+\frac1{2h}\|U_j-U_{j-1}\|_{L^2(\rho\,\mathrm ds)}^2
        \leq\mathcal E(U_{j-1}).
    \end{equation*}
    Thus the piecewise linear time interpolants are bounded in
    $L^\infty(0,T;H^1(\ell,L))\cap H^1(0,T;L^2(\ell,L))$.
    Compactness of $H^1(\ell,L)\hookrightarrow L^2(\ell,L)$ and the resulting uniform $1/2$-H\"older modulus in time in $L^2$ yield a subsequence converging in $C([0,T];L^2(\ell,L))$. The piecewise constant interpolants have the same limit. Passing to the integrated discrete variational inequalities, using weak convergence of the derivatives, strong $L^2$ convergence, and lower semicontinuity of the quadratic energy, identifies the limit with the parabolic variational solution with obstacle $\phi$ and these boundary values. Uniqueness follows by testing the difference of two such inequalities with the difference of their solutions.
    This solution is $v(s,\tau):=Y(\ln s,\tau)$ restricted to $[\ell,L]$. Indeed, Lemma \ref{lemma_nonempty} gives $v(\ell,\tau)=D$, and the strong obstacle inequalities give the same variational inequality. The bounds $0\leq v_s\leq1$ give a uniform bound on $\mathcal E(v(\cdot,\tau))$ down to $\tau=0$. On $[\varepsilon,T]$ the usual energy identity is justified by the strong regularity, or by time difference quotients: the obstacle residual times $v_\tau$ is zero almost everywhere, since the residual vanishes in continuation and $v_\tau=0$ on contact. Therefore
    \begin{equation*}
        \mathcal E(v(\cdot,T))+
        \int_\varepsilon^T\|v_\tau\|_{L^2(\rho\,\mathrm ds)}^2\,\mathrm d\tau
        =\mathcal E(v(\cdot,\varepsilon)).
    \end{equation*}
    Letting $\varepsilon\searrow0$ proves that $v$ belongs to the energy class in which uniqueness was just established, including at the initial payoff corner.
    Each time interpolant is convex in $s$. The convergence in $C([0,T];L^2)$ preserves nonnegativity of the spatial second derivative as a distribution at every time. Since the limit is continuous, it is convex in $s$. Its constant extension on $(0,\ell)$ is already the original solution, which proves the lemma.
\end{proof}

\medskip
    Condition \eqref{cond_convex_cap} cannot be replaced by $K_1+k_1\geq K_2$ alone. Indeed, suppose that $q=0$ and $c:=K_1+k_1-K_2>0$. By Theorem \ref{b0}, continuation reaches $L:=k_1$ for every positive time. For $v(s,\tau):=Y(\ln s,\tau)$ and $M:=L-c$, regularity at the fixed Dirichlet boundary gives
    \begin{equation*}
        v_{ss}(L-,\tau)=\frac{2r}{\sigma^2L^2}
        \left(M-Lv_s(L-,\tau)\right).
    \end{equation*}
    The initial payoff is $s-c$ near $L$ and agrees with the Dirichlet value there. The half-line regularity at this corner therefore gives continuity of the spatial first derivative and $v_s(L-,\tau)\to1$ as $\tau\searrow0$, while no compatibility of the second derivatives is asserted. Consequently $v_{ss}(L-,\tau)\to-2rc/(\sigma^2L^2)<0$, and convexity fails near the cap for sufficiently small positive times.

\begin{theorem} \label{thm:prop_freebdry}
    Suppose that \eqref{cond_convex_cap} holds. For fixed $r,q$ and strike prices, the lower free boundary $A(\tau)$ is nonincreasing and the upper free boundary $B(\tau)$ is nondecreasing with respect to $\sigma$, for every $\tau\in(0,T]$.
\end{theorem}
\begin{proof}
    Let $0<\sigma_2<\sigma_1$ and denote the corresponding solutions by $Y_2,Y_1$. Write $\mathcal L_{\sigma_j}$ for \eqref{operator_L} with volatility $\sigma_j$. Lemma \ref{lem3.1} gives
    \begin{equation*}
        (\partial_\tau-\mathcal L_{\sigma_2})Y_1
        =(\partial_\tau-\mathcal L_{\sigma_1})Y_1
        +\frac{\sigma_1^2-\sigma_2^2}{2}
        (\partial_{xx}Y_1-\partial_xY_1)\geq0
    \end{equation*}
    almost everywhere. The obstacle and the initial and boundary values are independent of $\sigma$. Thus $Y_1$ is a supersolution for the problem defining $Y_2$, and Lemma \ref{lem_comparison} gives $Y_1\geq Y_2$. Since both functions dominate the same obstacle, the exercise region of $Y_1$ is contained in that of $Y_2$. The descriptions of their lower and upper exercise intervals therefore imply
    \begin{equation*}
        A_1(\tau)\leq A_2(\tau),\qquad
        B_1(\tau)\geq B_2(\tau),\qquad \tau\in(0,T].
    \end{equation*}
\end{proof}

\section{Stationary problem for the condor spread} \label{sec5}
In this section, we study the stationary problems that arise from the condor spread, which correspond to perpetual contracts. Their solutions are explicit, and they give bounds for the free boundaries of Section \ref{sec4} that are uniform in time. Throughout, $ n_1<0<n_2 $ denote the roots of the characteristic equation
\begin{equation} \label{char_eq}
        \frac{\sigma^{2}}{2} n^{2}+\left(r-q-\frac{\sigma^{2}}{2}\right) n-r=0,
\end{equation}
so that $\mathcal{L}e^{n_1x}=\mathcal{L}e^{n_2x}=0$. If $q=0$, then $n_1=-\frac{2r}{\sigma^2}$ and $n_2=1$, and if $q>0$, then $n_2>1$. Moreover
\begin{equation} \label{sta_cal}
    \frac{n_1}{n_1-1}\cdot \frac{n_2}{n_2-1}=\frac{r}{q}, \qquad \frac{n_2(n_1-1)}{n_2-1} \leq n_1 \qquad \text{if }q>0,
\end{equation}
which follow from $n_1n_2=-\frac{2r}{\sigma^2}$ and $(n_1-1)(n_2-1)=-\frac{2q}{\sigma^2}$.

The stationary problem corresponding to \eqref{condor_spread_left} is
\begin{equation} \label{stationary_prob}
    \begin{cases}
        -\mathcal{L} V=0, & \text { if } V>\psi_1(x), \qquad x \in (-\infty,\ln k_1),\\[0.3cm]
        -\mathcal{L} V \geq 0, & \text { if } V=\psi_1(x), \qquad x \in (-\infty,\ln k_1),\\[0.3cm]
        V(\ln k_1)=K_2-K_1, &
    \end{cases}
\end{equation}
and the stationary problem corresponding to \eqref{uncapped_left} is
\begin{equation} \label{stationary_uncapped}
    \begin{cases}
        -\mathcal{L} V=0, & \text { if } V>\psi_1(x), \qquad x \in \mathbb{R},\\[0.3cm]
        -\mathcal{L} V \geq 0, & \text { if } V=\psi_1(x), \qquad x \in \mathbb{R}.
    \end{cases}
\end{equation}
A strong solution of \eqref{stationary_prob} is a bounded function $V\in C((-\infty,\ln k_1])\cap W^{2}_{p,\rm{loc}}((-\infty,\ln k_1))$ with $V\geq\psi_1$, $-\mathcal{L}V\geq0$ almost everywhere, $-\mathcal{L}V=0$ almost everywhere on $\{V>\psi_1\}$ and $V(\ln k_1)=K_2-K_1$. For \eqref{stationary_uncapped} we require $V\in C(\mathbb{R})\cap W^{2}_{p,\rm{loc}}(\mathbb{R})$, the same obstacle and differential conditions on $\mathbb R$, and the natural price bounds
\begin{equation*}
    \psi_1(x)\leq V(x)\leq e^x+K_2-k_1,\qquad x\in\mathbb R.
\end{equation*}
For stopping values, the upper bound follows from the supermartingale property of $e^{-rt}S_t$ and from the bound $K_2-k_1$ on the cash component. It is essential to the stationary characterization: linear growth alone need not give uniqueness when $q=0$. Regarded as functions of $(x,\tau)$, these stationary solutions are supersolutions of \eqref{condor_spread_left} and \eqref{uncapped_left}, respectively, so the comparison Lemma \ref{lem_comparison} (i) applies to them.
\begin{lemma} \label{lem_stationary_unique}
    Each of \eqref{stationary_prob} and \eqref{stationary_uncapped} has at most one strong solution in its specified class.
\end{lemma}
\begin{proof}
    First let $V_1,V_2$ solve \eqref{stationary_prob}, put $U:=V_1-V_2$ and let $U_\delta:=U-\delta e^{n_1x}$ for $\delta>0$. On the set $\{U>0\}$ we have $V_1>V_2\geq\psi_1$, so $-\mathcal{L}V_1=0\leq-\mathcal{L}V_2$ and hence $-\mathcal{L}U\leq0$ almost everywhere there. Since $\mathcal{L}e^{n_1x}=0$, also $-\mathcal{L}U_\delta\leq0$ almost everywhere on $\{U_\delta>0\}\subset\{U>0\}$. As $U$ is bounded and $e^{n_1x}\to\infty$ as $x\to-\infty$, the open set $\{U_\delta>0\}$ is bounded, and $U_\delta\leq U(\ln k_1)-\delta k_1^{n_1}<0$ at $x=\ln k_1$. The maximum principle for strong subsolutions, see \cite[Theorem 9.1]{gilbarg2001elliptic}, gives $U_\delta\leq0$. Letting $\delta\searrow0$ yields $U\leq0$, and by symmetry $U=0$.

    For \eqref{stationary_uncapped}, the price bounds imply
    \begin{equation*}
        |V_1(x)-V_2(x)|\leq e^x+K_2-k_1-\psi_1(x)
        =\min\{e^x,K_1\}\leq K_1.
    \end{equation*}
    Use instead $U_\delta:=U-\delta(e^{n_1x}+e^{n_2x})$. The subtracted function is $\mathcal L$-harmonic and diverges at both infinite endpoints. Thus $\{U_\delta>0\}$ is bounded, and the same maximum-principle argument gives $U\leq0$ and then equality. This also covers $q=0$, when $n_2=1$.
\end{proof}

We begin with the uncapped problem \eqref{stationary_uncapped}, whose solution has two free boundaries $a<\ln K_1<b$ when the call side admits early exercise.
\begin{theorem} \label{theorem3.1}
    Assume that $q>0$, or that $q=0$ and $K_1+k_1<K_2$. Let $y>1$ be the unique solution of $f(y)=0$, where
    \begin{equation} \label{def_f}
        f(z)=\frac{K_2-k_1}{n_2-n_1}\left[n_2(1-n_1)z^{n_1}+n_1(n_2-1)z^{n_2}\right]+(K_1+k_1-K_2),
    \end{equation}
    and let $a<b$ be defined by
    \begin{equation} \label{def_ab}
        e^b=\frac{(K_2-k_1)n_1n_2}{n_2-n_1}\left[y^{n_1}-y^{n_2}\right] , \qquad
        e^a=\frac{e^b}{y}.
    \end{equation}
    Then $a<\ln K_1<b$, and the function
    \begin{equation} \label{stationary_prob_V}
        V^{\infty}(x)=
        \begin{cases}
        \displaystyle  K_2-k_1, &   x < a,\\[0.45cm]
        \displaystyle  \frac{K_2-k_1}{n_2-n_1}\left[n_2e^{n_1(x-a)}-n_1e^{n_2(x-a)}\right], & a \leq x \leq b , \\[0.45cm]
        \displaystyle  e^x-K_1+K_2-k_1, &   x > b,
        \end{cases}
    \end{equation}
    is the unique strong solution of \eqref{stationary_uncapped}. Its coincidence set is $(-\infty,a]\cup[b,\infty)$, and $V^\infty-\psi_1$ is strictly decreasing on $[\ln K_1,b]$.
\end{theorem}
\begin{proof}
    First, consider the free boundary problem
    \begin{equation} \label{free_boundary_stationary}
    \begin{cases}
        -\mathcal{L}V=0, \qquad a \leq x \leq b,\\[0.3cm]
        V(a)=K_2-k_1, \qquad V^{\prime}(a)=0, \\[0.3cm]
        V(b)=e^b-K_1+K_2-k_1, \qquad V^{\prime}(b)=e^b,
    \end{cases}
    \end{equation}
    in which $a<b$ are unknown. The general solution of $-\mathcal{L}V=0$ is $V(x)=c_1e^{n_1x}+ c_2 e^{n_2x}$, and the conditions at $x=a$ give
    \begin{equation} \label{c_1c_2}
            c_1=\frac{(K_2-k_1)n_2}{n_2-n_1}e^{-n_1a},\qquad
            c_2=-\frac{(K_2-k_1)n_1}{n_2-n_1}e^{-n_2a}.
    \end{equation}
    Hence
    \begin{equation} \label{Vprime}
        V(x)=\frac{K_2-k_1}{n_2-n_1}\left[n_2e^{n_1(x-a)}-n_1e^{n_2(x-a)}\right], \qquad V^{\prime}(x)=\frac{(K_2-k_1)n_1n_2}{n_2-n_1}\left[e^{n_1(x-a)}-e^{n_2(x-a)}\right].
    \end{equation}
    Let $ y=e^{b-a}>1 $. The conditions at $x=b$ read $V^{\prime}(b)=e^b$ and $V(b)=V^{\prime}(b)+(K_2-K_1-k_1)$. Substituting \eqref{Vprime}, the second condition becomes $f(y)=0$ with $f$ given by \eqref{def_f}, and the first one gives \eqref{def_ab}. We have
    \begin{equation*}
        f^{\prime}(z)=\frac{(K_2-k_1)n_1n_2}{n_2-n_1}\left[(1-n_1)z^{n_1-1}+(n_2-1)z^{n_2-1}\right] < 0 , \qquad  1<z<+\infty,
    \end{equation*}
    since $n_1n_2<0$, and
    \begin{equation*}
        f(1)=(K_2-k_1)+(K_1+k_1-K_2)=K_1 >0
    \end{equation*}
    and
    \begin{equation*}
        \lim_{z\to\infty}f(z)=
        \begin{cases}
        -\infty, & \text{ if } q >0, \\[0.2cm]
        K_1+k_1-K_2<0, & \text{ if } q=0 \text{ and } K_1+k_1<K_2.
        \end{cases}
    \end{equation*}
    Therefore $ f(y)=0 $ has a unique solution $ y \in (1,+\infty)$. Since $n_1<0<n_2$ and $y>1$, the right-hand side of the first identity in \eqref{def_ab} is positive, so $b$ and $a=b-\ln y$ are well defined and $a<b$.

    Next we verify that $V^\infty$ defined by \eqref{stationary_prob_V} is a strong solution of \eqref{stationary_uncapped}. It is continuous, and it is $C^1$ at $a$ and at $b$ by construction, so $V^\infty\in W^2_{p,\rm{loc}}(\mathbb{R})$. We have $-\mathcal{L}V^\infty=r(K_2-k_1)>0$ on $(-\infty,a)$ and $-\mathcal{L}V^\infty=0$ on $(a,b)$. From \eqref{Vprime} we obtain $ (V^{\infty})^{\prime} \geq 0 $ on $[a,b]$, so $V^\infty\geq K_2-k_1$ on $[a,b]$. Moreover
    \begin{equation*}
     \begin{cases}
     -\mathcal{L}\left[(V^{\infty})^{\prime}-e^x \right]=\mathcal{L}\left[e^x\right]=-qe^x \leq 0,& a<x<b ,\\[0.3cm]
     (V^{\infty})^{\prime}(a)-e^a=-e^a < 0,\\[0.3cm]
     (V^{\infty})^{\prime}(b)-e^b=0,
    \end{cases}
    \end{equation*}
    so the maximum principle gives $(V^{\infty})^{\prime} \leq e^x$ on $[a,b]$, with strict inequality on $[a,b)$ by the strong maximum principle. Since $V^\infty(b)=e^b-K_1+K_2-k_1$, integration from $x$ to $b$ gives
    \begin{equation*}
        V^\infty(x)>e^x-K_1+K_2-k_1,\qquad a\leq x<b.
    \end{equation*}
    At $x=a$, the equality $V^\infty(a)=K_2-k_1$ therefore implies $e^a<K_1$. Also, \eqref{Vprime} gives $(V^\infty)'(x)>0$ for $a<x\leq b$, so
    \begin{equation*}
        e^b-K_1+K_2-k_1=V^\infty(b)>V^\infty(a)=K_2-k_1,
    \end{equation*}
    and hence $e^b>K_1$. This proves $a<\ln K_1<b$ before identifying the contact regions. The preceding inequalities now give $V^\infty\geq\psi_1$ on $\mathbb R$, with coincidence set $(-\infty,a]\cup[b,\infty)$, and $V^\infty-\psi_1$ is strictly decreasing on $[\ln K_1,b]$. It remains to show that $-\mathcal L\psi_1\geq0$ on $(b,\infty)$, that is, by \eqref{Lpsi1},
    \begin{equation*}
        qe^x-r(K_1+k_1-K_2)\geq0, \qquad x>b.
    \end{equation*}
    If $K_1+k_1\leq K_2$ this is obvious. If $q>0$ and $K_2 < K_1+k_1$, it suffices to show that $e^b \geq \frac{r(K_1+k_1-K_2)}{q}$. From \eqref{def_ab}, the second inequality in \eqref{sta_cal}, $f(y)=0$ and the first identity in \eqref{sta_cal}, we deduce that
    \begin{align*}
        e^b &=\frac{(K_2-k_1)n_1n_2}{n_2-n_1}\left(y^{n_1}-y^{n_2}\right)\\[0.3cm]
        &\geq \frac{n_2(K_2-k_1)}{n_2-n_1}\left[\frac{n_2(n_1-1)}{n_2-1}y^{n_1}-n_1y^{n_2}\right] \\[0.3cm]
        &= \frac{n_2}{n_2-1}\cdot \frac{K_2-k_1}{n_2-n_1}\left[n_2(n_1-1)y^{n_1}-n_1(n_2-1)y^{n_2}\right] \\[0.3cm]
        & = \frac{n_2}{n_2-1}(K_1+k_1-K_2) \\[0.3cm]
        &\geq \frac{n_1}{n_1-1}\cdot\frac{n_2}{n_2-1}(K_1+k_1-K_2) \\[0.3cm]
        & = \frac{r(K_1+k_1-K_2)}{q},
    \end{align*}
    where we also used $0<\frac{n_1}{n_1-1}<1$. Finally, the derivative bound $(V^\infty)'\leq e^x$ holds on all of $\mathbb R$. Integrating it from $a$ gives $V^\infty(x)\leq K_2-k_1+e^x-e^a<e^x+K_2-k_1$ for $x\geq a$, and the required upper bound is immediate for $x<a$. Thus $V^\infty$ belongs to the specified price class, and uniqueness follows from Lemma \ref{lem_stationary_unique}.
\end{proof}

In the remaining case $q=0$ and $K_1+k_1\geq K_2$, put
\begin{equation*}
    D:=K_2-k_1,\qquad \alpha:=\frac{2r}{\sigma^2},\qquad
    s_*:=\frac{D\alpha}{1+\alpha},\qquad a_*:=\ln s_*.
\end{equation*}
Then $s_*<D\leq K_1$, and the unique strong solution of \eqref{stationary_uncapped} is
\begin{equation*}
    V^\infty(x)=
    \begin{cases}
        D,&x\leq a_*,\\[0.25cm]
        \displaystyle e^x+\frac{D}{1+\alpha}e^{-\alpha(x-a_*)},&x>a_*.
    \end{cases}
\end{equation*}
Indeed, value matching and zero derivative hold at $a_*$, and the two terms above $a_*$ are $\mathcal L$-harmonic because $n_1=-\alpha$ and $n_2=1$. The expression is strictly increasing for $x>a_*$ and strictly exceeds $e^x$, whereas $\psi_1(x)=\max\{D,e^x-(K_1-D)\}$. It therefore strictly dominates the obstacle on $(a_*,\infty)$ and is at most $e^x+D$. Below $a_*$ its residual is $rD>0$. Uniqueness follows from Lemma \ref{lem_stationary_unique}, and the coincidence set is $(-\infty,a_*]$. In this case we set $b=+\infty$.

The classification of the stationary problem \eqref{stationary_prob} compares the upper free boundary of \eqref{stationary_uncapped} with the cap $\ln k_1$, as in Theorem \ref{thm_reduction}. If $b\leq\ln k_1$, the restriction of $V^\infty$ solves \eqref{stationary_prob}. If $b>\ln k_1$, or if the uncapped upper boundary is infinite, the capped problem has only a lower interior free boundary, which must be determined again from the boundary value at $\ln k_1$.
\begin{theorem} \label{theorem3.2}
    \begin{enumerate}[label=(\roman*)]
        \item Under the assumptions of Theorem \ref{theorem3.1}, suppose that $b\leq\ln k_1$. Then the restriction of $V^\infty$ to $(-\infty,\ln k_1]$ is the unique strong solution of \eqref{stationary_prob}.
        \item Suppose that $q=0$ and $K_1+k_1\geq K_2$, or that the assumptions of Theorem \ref{theorem3.1} hold and $b>\ln k_1$. Let $y>1$ be the unique solution of $g(y)=0$, where
        \begin{equation} \label{def_g}
            g(z)=\frac{K_2-k_1}{n_2-n_1}\left[n_2z^{n_1}-n_1z^{n_2}\right]-(K_2-K_1),
        \end{equation}
        and let $e^{a}=k_1/y$. Then the unique strong solution of \eqref{stationary_prob} is
        \begin{equation} \label{stationary_prob_V_2}
            V(x)=
            \begin{cases}
            \displaystyle  K_2-k_1, & x<a,\\[0.45cm]
            \displaystyle  \frac{K_2-k_1}{n_2-n_1}\left[n_2e^{n_1(x-a)}-n_1e^{n_2(x-a)}\right], & a \leq x\leq \ln k_1,
            \end{cases}
        \end{equation}
        and its coincidence set is $(-\infty,a]\cup\{\ln k_1\}$.
    \end{enumerate}
\end{theorem}
\begin{proof}
    The uniqueness is Lemma \ref{lem_stationary_unique}. For (i), if $b\leq\ln k_1$, then $\ln k_1$ belongs to the coincidence set of $V^\infty$, so $V^\infty(\ln k_1)=\psi_1(\ln k_1)=K_2-K_1$, and the restriction of $V^\infty$ is bounded and satisfies all conditions in \eqref{stationary_prob}.

    For (ii), the function $g$ satisfies
    \begin{equation*}
        g(1)=(K_2-k_1)-(K_2-K_1)=K_1-k_1 <0,\qquad
        \lim_{z\to\infty}g(z)=+\infty
    \end{equation*}
    and
    \begin{equation*}
        g^{\prime}(z)=\frac{(K_2-k_1)n_1n_2}{n_2-n_1}\left[z^{n_1-1}-z^{n_2-1} \right] > 0 \quad \text{ for } z>1,
    \end{equation*}
    so $g(y)=0$ has a unique solution $y\in(1,\infty)$, and $a=\ln k_1-\ln y<\ln k_1$. The function $V$ in \eqref{stationary_prob_V_2} is $C^1$ at $a$, it satisfies $-\mathcal{L}V=r(K_2-k_1)>0$ on $(-\infty,a)$, $-\mathcal{L}V=0$ on $(a,\ln k_1)$, and $V(\ln k_1)=K_2-K_1$ by the choice of $y$. As in the proof of Theorem \ref{theorem3.1}, $V^{\prime}\geq0$ on $[a,\ln k_1]$, so $V\geq K_2-k_1$. It remains to show that $V\geq e^x-K_1+K_2-k_1$ on $[\ln K_1,\ln k_1]$, and since $V(\ln k_1)=\psi_1(\ln k_1)$ it suffices to prove $V'\leq e^x$ on $(a,\ln k_1)$. The function $V'-e^x$ satisfies $-\mathcal{L}(V'-e^x)=-qe^x\leq0$ on $(a,\ln k_1)$ and $V'(a)-e^a<0$, so by the maximum principle it suffices to show that
    \begin{equation} \label{Vprime_at_cap}
        V'(\ln k_1)\leq k_1.
    \end{equation}
    Suppose first that $q=0$ and $K_1+k_1\geq K_2$. Then $n_2=1$, and \eqref{stationary_prob_V_2} gives the identity $V(x)=V^{\prime}(x)+(K_2-k_1)e^{n_1(x-a)}$ on $[a,\ln k_1]$, whence
    \begin{equation*}
        V'(\ln k_1)=K_2-K_1-(K_2-k_1)y^{n_1}<K_2-K_1\leq k_1.
    \end{equation*}
    Suppose next that the assumptions of Theorem \ref{theorem3.1} hold and $b>\ln k_1$. Write
    \begin{equation*}
        G(u):=\frac{K_2-k_1}{n_2-n_1}\left[n_2e^{n_1u}-n_1e^{n_2u}\right], \qquad u\geq0,
    \end{equation*}
    so that $V(x)=G(x-a)$ on $[a,\ln k_1]$ and $V^\infty(x)=G(x-a^\infty)$ on $[a^\infty,b]$, where $a^\infty$ denotes the lower free boundary of Theorem \ref{theorem3.1}. The function $G$ is increasing and convex on $[0,\infty)$, because
    \begin{equation*}
        G'(u)=\frac{(K_2-k_1)n_1n_2}{n_2-n_1}\left[e^{n_1u}-e^{n_2u}\right]>0 \qquad \text{ and } \qquad G''(u)=\frac{(K_2-k_1)n_1n_2}{n_2-n_1}\left[n_1e^{n_1u}-n_2e^{n_2u}\right]>0
    \end{equation*}
    for $u>0$. Since $\ln k_1<b$, the point $\ln k_1$ lies in the continuation set of $V^\infty$, so $G(\ln k_1-a^\infty)=V^\infty(\ln k_1)>\psi_1(\ln k_1)=K_2-K_1=G(\ln k_1-a)$, and the monotonicity of $G$ gives $a^\infty<a$. Therefore, by the convexity of $G$ and the inequality $(V^\infty)'\leq e^x$ on $[a^\infty,b]$ established in the proof of Theorem \ref{theorem3.1},
    \begin{equation*}
        V'(\ln k_1)=G'(\ln k_1-a)\leq G'(\ln k_1-a^\infty)=(V^\infty)'(\ln k_1)\leq k_1,
    \end{equation*}
    which is \eqref{Vprime_at_cap}. The strong maximum principle gives $V'(x)<e^x$ for $a\leq x<\ln k_1$, while the explicit formula gives $V'(x)>0$ for $a<x\leq\ln k_1$. Integrating the first inequality from $a$ to $\ln k_1$ yields
    \begin{equation*}
        V(a)>e^a-K_1+K_2-k_1.
    \end{equation*}
    Since $V(a)=K_2-k_1$, this proves $a<\ln K_1$. The two strict derivative inequalities and the endpoint values then give $V>\psi_1$ on $(a,\ln k_1)$ and $V=\psi_1$ on $(-\infty,a]\cup\{\ln k_1\}$. Thus $V$ is the asserted strong solution.
\end{proof}

\begin{remark}
    Theorem \ref{theorem3.2} says that the upper free boundary of the stationary problem \eqref{stationary_prob} is $\min\{b,\ln k_1\}$, with $b=+\infty$ when $q=0$ and $K_1+k_1\geq K_2$. This is the stationary counterpart of Theorem \ref{thm_reduction}. Note also that the condition $b\leq\ln k_1$ in (i) is a condition on the parameters, since $b$ is determined by \eqref{def_f} and \eqref{def_ab}.
\end{remark}

Using the comparison principle, we obtain bounds for the free boundaries $A$ and $B$ that are uniform in time.
\begin{proposition} \label{prop_YV}
    Let $V$ be the strong solution of \eqref{stationary_prob} and let $a$ be its lower free boundary, given by Theorem \ref{theorem3.2}. Let $b_{k_1}:=b$ in the case (i) of Theorem \ref{theorem3.2} and $b_{k_1}:=\ln k_1$ in the case (ii). Then
    \begin{equation*}
     \begin{cases}
        \displaystyle  Y(x,\tau) \leq V(x),&  (x,\tau) \in (-\infty,\ln k_1] \times [0,T],\\[0.3cm]
        \displaystyle  a \leq A(\tau) < \ln K_1, & \tau \in (0,T], \\[0.3cm]
        \displaystyle  \min\{\xi,\ln k_1\} \leq B(\tau) \leq b_{k_1}, & \tau \in (0,T],
    \end{cases}
    \end{equation*}
    where $\xi$ is defined in \eqref{def_xi}. In particular $A(\tau)\searrow A(T)\geq a$ and $B(\tau)\nearrow B(T)\leq b_{k_1}$ as $\tau\nearrow T$, and the limits as $T\to\infty$ exist.
\end{proposition}
\begin{proof}
    Regarded as a function of $(x,\tau)$, $V$ is a bounded supersolution of \eqref{condor_spread_left}, since $V\geq\psi_1$, $\partial_\tau V-\mathcal{L}V=-\mathcal{L}V\geq0$ and $V(\ln k_1)=K_2-K_1$. Lemma \ref{lem_comparison} (i) gives $Y\leq V$. Hence $Y=\psi_1$ wherever $V=\psi_1$, that is, $Y(x,\tau)=K_2-k_1$ for $x\leq a$ and $Y(x,\tau)=\psi_1(x)$ for $b_{k_1}\leq x\leq\ln k_1$. This gives $a\leq A(\tau)$ and $B(\tau)\leq b_{k_1}$. The strict upper bound for $A$ follows from Lemma \ref{lemma_K_1} and Theorem \ref{a0}. For the lower bound for $B$, Theorem \ref{b0} gives $B=\ln k_1$ in case (i), and $B(\tau)\geq B(0)=\xi$ in case (ii), so that $B(\tau)\geq\min\{\xi,\ln k_1\}$ in all cases. The last assertion follows from the monotonicity of $A$ and $B$ in $\tau$, which does not depend on $T$ because the solution of \eqref{condor_spread_left} with a larger horizon restricts to the solution with a smaller horizon.
\end{proof}

We finally consider the problem \eqref{condor_spread_right}. Its stationary problem is
\begin{equation} \label{stationary_prob_2}
    \begin{cases}
        -\mathcal{L} V=0, & \text { if } V>\psi_2(x), \;\;x \in (\ln k_2,+\infty),\\[0.3cm]
        -\mathcal{L} V \geq 0, & \text { if } V=\psi_2(x), \;\; x \in (\ln k_2,+\infty),\\[0.3cm]
        V(\ln k_2)=K_2-K_1, &
    \end{cases}
\end{equation}
and the uncapped stationary problem is the same problem on $\mathbb R$ without the boundary condition. Here a strong solution of \eqref{stationary_prob_2} means a bounded function
\begin{equation*}
    V\in C([\ln k_2,\infty))\cap W^2_{p,\mathrm{loc}}((\ln k_2,\infty))
\end{equation*}
with $V\geq\psi_2$, $-\mathcal LV\geq0$ almost everywhere, $-\mathcal LV=0$ almost everywhere on $\{V>\psi_2\}$, and the prescribed boundary value. For the uncapped problem we impose the same obstacle and differential conditions on $\mathbb R$ and again require boundedness. The bounded class is essential for the half-line uniqueness argument: the proof of Lemma \ref{lem_stationary_unique} applies with the barrier $e^{n_2x}$, which diverges as $x\to+\infty$. The formulas below require no distinction between $q=0$ and $q>0$, because $-\mathcal L\psi_2=-qe^x+r(K_2+k_2-K_1)$ is strictly positive on $(\ln k_2,\ln K_2)$ when $q=0$, so that the put side always admits early exercise.
\begin{theorem} \label{theorem3.3}
    Let $\widetilde{y}\in(0,1)$ be the unique solution of $\widetilde{f}(\widetilde{y})=0$, where
    \begin{equation*}
        \widetilde{f}(z)=\frac{k_2-K_1}{n_2-n_1}\left[n_2(1-n_1)z^{n_1}+n_1(n_2-1)z^{n_2}\right]-(K_2+k_2-K_1),
    \end{equation*}
    and let $\widetilde{b}<\widetilde{a}$ be defined by
    \begin{equation*}
        e^{\widetilde{b} }=-\frac{(k_2-K_1)n_1n_2}{n_2-n_1}\left[\widetilde{y}^{\,n_1}-\widetilde{y}^{\,n_2}\right], \qquad
        e^{\widetilde{a} }=\frac{e^{\widetilde{b} }}{\widetilde{y}}.
    \end{equation*}
    Then $\widetilde{b}<\ln K_2<\widetilde{a}$, and the function
    \begin{equation} \label{stationary_prob_2_V}
        \widetilde{V}^\infty(x)=
        \begin{cases}
        \displaystyle  K_2-e^x+k_2-K_1, &   x < \widetilde{b} ,\\[0.45cm]
        \displaystyle  \frac{k_2-K_1}{n_2-n_1}\left[n_2e^{n_1(x-\widetilde{a} )}-n_1e^{n_2(x-\widetilde{a} )}\right], & \widetilde{b}  \leq x \leq \widetilde{a}  , \\[0.45cm]
        \displaystyle  k_2-K_1, &   x > \widetilde{a} ,
        \end{cases}
    \end{equation}
    is a strong solution of the uncapped stationary problem on $\mathbb{R}$ with the obstacle $\psi_2$. If $\widetilde{b}\geq\ln k_2$, then the restriction of $\widetilde{V}^\infty$ to $[\ln k_2,\infty)$ is the unique strong solution of \eqref{stationary_prob_2}. If $\widetilde{b}<\ln k_2$, then the unique strong solution of \eqref{stationary_prob_2} is
    \begin{equation*}
        \widetilde{V}(x)=
        \begin{cases}
        \displaystyle  \frac{k_2-K_1}{n_2-n_1}\left[n_2e^{n_1(x-\widetilde{a} )}-n_1e^{n_2(x-\widetilde{a} )}\right], & \ln k_2  \leq x \leq \widetilde{a}  , \\[0.45cm]
        \displaystyle  k_2-K_1, &   x > \widetilde{a} ,
        \end{cases}
    \end{equation*}
    where now $\widetilde{a}$ is determined by $\widetilde{V}(\ln k_2)=K_2-K_1$, that is, $e^{\widetilde a}=k_2/\widetilde{y}$ with the unique root $\widetilde{y}\in(0,1)$ of $\widetilde{g}(z)=\frac{k_2-K_1}{n_2-n_1}\left[n_2z^{n_1}-n_1z^{n_2}\right]-(K_2-K_1)$.
\end{theorem}
\begin{proof}
    The smooth fit conditions at $\widetilde b$ are $\widetilde V(\widetilde b)=K_2-e^{\widetilde b}+k_2-K_1$ and $\widetilde V'(\widetilde b)=-e^{\widetilde b}$. Together with value matching and zero derivative at $\widetilde a$, they give the displayed formulas and $\widetilde f(\widetilde y)=0$, where $\widetilde y=e^{\widetilde b-\widetilde a}$. We have $\widetilde f(z)\to+\infty$ as $z\searrow0$, $\widetilde f(1)=-K_2<0$, and $\widetilde f'<0$ on $(0,1)$. Thus the root is unique, and $e^{\widetilde b}>0$ because $\widetilde y^{n_1}>\widetilde y^{n_2}$.

    The continuation expression is strictly decreasing and strictly convex on $(\widetilde b,\widetilde a)$. Its derivative satisfies
    \begin{equation*}
        -\mathcal L\bigl[(\widetilde V^\infty)'+e^x\bigr]=qe^x\geq0,
    \end{equation*}
    with boundary values $0$ at $\widetilde b$ and $e^{\widetilde a}>0$ at $\widetilde a$. The strong maximum principle gives $(\widetilde V^\infty)'+e^x>0$ in $(\widetilde b,\widetilde a]$. Integrating from $\widetilde b$ yields
    \begin{equation*}
        \widetilde V^\infty(x)>K_2-e^x+k_2-K_1,
        \qquad \widetilde b<x\leq\widetilde a.
    \end{equation*}
    At $x=\widetilde a$ this proves $e^{\widetilde a}>K_2$. Strict decrease also gives
    $\widetilde V^\infty(\widetilde b)>\widetilde V^\infty(\widetilde a)=k_2-K_1$, so $e^{\widetilde b}<K_2$. Consequently $\widetilde b<\ln K_2<\widetilde a$, and the constructed function dominates $\psi_2$, strictly between its two free boundaries.

    On $(-\infty,\widetilde b)$ we must check $-\mathcal L\psi_2=-qe^x+r(K_2+k_2-K_1)\geq0$. This is immediate when $q=0$. When $q>0$, the equation $\widetilde f(\widetilde y)=0$ and \eqref{sta_cal} give the exact identity
    \begin{equation*}
        \frac{r(K_2+k_2-K_1)}q-e^{\widetilde b}
        =\frac{(k_2-K_1)(-n_1n_2)}{n_2-n_1}
        \left[
            \frac{\widetilde y^{n_1}}{n_2-1}
            +\frac{\widetilde y^{n_2}}{1-n_1}
        \right]>0.
    \end{equation*}
    On $(\widetilde a,\infty)$ the obstacle is the constant $k_2-K_1$, whose negative image under $\mathcal L$ is positive. Value matching and smooth fit therefore give the asserted bounded strong solution. If $\widetilde b\geq\ln k_2$, its restriction satisfies the boundary value in \eqref{stationary_prob_2}, and bounded uniqueness has already been established.

    Suppose now that $\widetilde b<\ln k_2$. The function $\widetilde g$ tends to $+\infty$ at $0$, satisfies $\widetilde g(1)=k_2-K_2<0$, and has negative derivative on $(0,1)$, so its root is unique. To verify the obstacle inequality for the resulting capped function, write
    \begin{equation*}
        G(u):=\frac{k_2-K_1}{n_2-n_1}
        \left[n_2e^{n_1u}-n_1e^{n_2u}\right],\qquad u\leq0,
    \end{equation*}
    which is strictly decreasing on $(-\infty,0)$ and strictly convex. Denote the uncapped upper boundary by $\widetilde a^\infty$ and the newly determined capped upper boundary by $\widetilde a$. Since
    \begin{equation*}
        G(\ln k_2-\widetilde a^\infty)
        =\widetilde V^\infty(\ln k_2)>K_2-K_1
        =G(\ln k_2-\widetilde a),
    \end{equation*}
    monotonicity gives $\widetilde a<\widetilde a^\infty$. Comparing these two upper free boundaries and using convexity yields
    \begin{equation*}
        \widetilde V'(\ln k_2)
        =G'(\ln k_2-\widetilde a)
        \geq G'(\ln k_2-\widetilde a^\infty)
        =(\widetilde V^\infty)'(\ln k_2)\geq-k_2.
    \end{equation*}
    Apply the maximum principle to $\widetilde V'+e^x$ on $(\ln k_2,\widetilde a)$, with these nonnegative left boundary data and the positive right boundary value $e^{\widetilde a}$. Integration from $\ln k_2$ proves strict domination of the affine part of $\psi_2$ up to $\widetilde a$, and at $\widetilde a$ it gives $e^{\widetilde a}>K_2$. Strict decrease toward the value $k_2-K_1$ gives domination of the constant part as well. The capped function is therefore a bounded strong solution, and uniqueness completes the proof.
\end{proof}
Similar to Proposition \ref{prop_YV}, we obtain the following bounds for the free boundaries of \eqref{condor_spread_right}, where $\widetilde{b}_{k_2}:=\widetilde{b}$ if $\widetilde{b}\geq\ln k_2$ and $\widetilde{b}_{k_2}:=\ln k_2$ otherwise.
\begin{proposition} \label{prop_YV2}
    Let $Y$ be the solution of \eqref{condor_spread_right} and let $\widetilde{V}$ be the strong solution of \eqref{stationary_prob_2}. Then
    \begin{equation*}
     \begin{cases}
        \displaystyle  Y(x,\tau) \leq \widetilde{V}(x),&  (x,\tau) \in [\ln k_2,+\infty)\times [0,T],\\[0.3cm]
        \displaystyle  \ln K_2 < \widetilde{A} (\tau) \leq \widetilde{a} , & \tau \in (0,T], \\[0.3cm]
        \displaystyle  \widetilde{b}_{k_2} \leq \widetilde{B} (\tau) \leq \widetilde{B} (0), & \tau \in (0,T].
    \end{cases}
    \end{equation*}
\end{proposition}
\begin{proof}
    The proof is the same as that of Proposition \ref{prop_YV}.
\end{proof}

\medskip
\noindent\emph{Perpetual interpretation.}
For a payoff $\Phi$, its perpetual value is
\begin{equation*}
    \sup_{\theta\in\mathcal T_{0,\infty}}
    \mathbb E_s\left[e^{-r\theta}\Phi(S_\theta)\right],
\end{equation*}
where $\mathcal T_{0,\infty}$ denotes the set of almost surely finite stopping times and $S_0=s$. The stationary functions above give this value for the uncapped payoffs $\Phi(s)=\psi_i(\ln s)$, $i=1,2$. For the full condor payoff $\Phi=V_0$, the stationary value $w$ is obtained by joining $V(\ln s)$ for $s<k_1$, the constant $M:=K_2-K_1$ on $[k_1,k_2]$, and $\widetilde V(\ln s)$ for $s>k_2$.

To verify these assertions, first let $w$ be either of the two uncapped stationary functions, written in the spot variable. The generalized It\^{o} formula, applied with localization, and the inequality $-\mathcal L V\geq0$ show that $e^{-rt}w(S_t)$ is a nonnegative local supermartingale, hence a supermartingale. Localization followed by Fatou's lemma therefore gives
\begin{equation*}
    \mathbb E_s\left[e^{-r\theta}\Phi(S_\theta)\right]\leq w(s)
    \qquad(\theta\in\mathcal T_{0,\infty}).
\end{equation*}
For the full condor, replacing $\theta$ by its minimum with the first entrance into the plateau can only increase the discounted payoff, since $V_0\leq M$. The same localization on the initial side of the plateau, stopped upon entering it, gives this upper bound, as in Theorem \ref{thm_verification}.
Except for the uncapped call with $q=0$ and $K_1+k_1\geq K_2$, every continuation interval has finite endpoints in the log variable. Its first exit $\theta^*$ is almost surely finite and is a contact time. The equation in continuation makes the stopped It\^{o} identity an equality. Since $w$ is bounded on the closed interval, the remaining term $\mathbb E_s[e^{-rt}w(S_t)\mathbf 1_{\{\theta^*>t\}}]$ is bounded by $Ce^{-rt}$ and tends to zero as $t\to\infty$. Thus
\begin{equation*}
    w(s)=\mathbb E_s\left[e^{-r\theta^*}\Phi(S_{\theta^*})\right],
\end{equation*}
with immediate stopping when $s$ is already in contact. This proves the perpetual identification in these cases, including the full condor payoff.

For the remaining uncapped call, use $D,\alpha,s_*$ defined above and put $c:=K_1-D\geq0$. If $s>s_*$ and $R>\max\{s,K_1\}$, let $\theta_R$ be the first exit from $(s_*,R)$. This time is almost surely finite, and the bounded-interval It\^{o} identity gives $w(s)=\mathbb E_s[e^{-r\theta_R}w(S_{\theta_R})]$. At the lower endpoint $w=\Phi$, while
\begin{equation*}
    w(R)-\Phi(R)=c+\frac{D}{1+\alpha}(R/s_*)^{-\alpha}.
\end{equation*}
Optional stopping for the nonnegative discounted stock gives $\mathbb E_s[e^{-r\theta_R}\mathbf1_{\{S_{\theta_R}=R\}}]\leq s/R$. Consequently
\begin{equation*}
    0\leq w(s)-\mathbb E_s[e^{-r\theta_R}\Phi(S_{\theta_R})]
    \leq\frac{s}{R}\left[c+\frac{D}{1+\alpha}(R/s_*)^{-\alpha}\right]
    \longrightarrow0.
\end{equation*}
These exit strategies approach the perpetual value as $R\to\infty$, and for $s\leq s_*$ immediate stopping gives the value. Thus the stationary price is identified also when its upper contact boundary is infinite.

\section{Numerical simulation} \label{sec6}

We compute the contract value in the original spot variable and use time to expiry throughout this section:
\begin{equation*}
    U(s,\tau):=V(s,T-\tau)=Y(\ln s,\tau),
    \qquad U(s,0)=V_0(s)=V(s,T).
\end{equation*}
The four boundaries in the figures are therefore $e^{A(\tau)}$, $e^{B(\tau)}$, $e^{\widetilde B(\tau)}$ and $e^{\widetilde A(\tau)}$, in increasing spatial order. Their values at $\tau=0$ denote the one-sided limits established in Section~\ref{sec4}, because the entire terminal slice is in contact.

We solve the two obstacle problems separately on $[s_{\min},k_1]$ and $[k_2,s_{\max}]$, with $s_{\min}=0.05$ and $s_{\max}=12$. We impose $U(k_1,\tau)=U(k_2,\tau)=M:=K_2-K_1$, set the value exactly equal to $M$ on the plateau, and impose the constant tail payoffs $K_2-k_1$ and $k_2-K_1$ at the outer computational endpoints. The truncation check below assesses these outer boundary choices. A uniform spot mesh aligns all four strikes. With $s_i=s_{\min}+i\Delta s$, $\tau_n=n\Delta\tau$, and $\phi_i=V_0(s_i)$, the backward Euler approximation is
\begin{equation} \label{numerical_lcp}
    \min\left\{\frac{U_i^n-U_i^{n-1}}{\Delta\tau}
                  -(\mathcal G_hU^n)_i,\ U_i^n-\phi_i\right\}=0,
    \qquad U_i^0=\phi_i,
\end{equation}
at the interior nodes of each half-line truncation, where
\begin{equation*}
    (\mathcal G_h u)_i=
    \frac{\sigma^2s_i^2}{2}\frac{u_{i+1}-2u_i+u_{i-1}}{\Delta s^2}
    +(r-q)s_i\frac{u_{i+1}-u_{i-1}}{2\Delta s}-ru_i.
\end{equation*}
The off-diagonal coefficients of $\mathcal G_h$ are nonnegative whenever $\sigma^2s_i\geq|r-q|\Delta s$, which holds on every grid used here. Thus $I-\Delta\tau\mathcal G_h$ has the monotonicity needed for the discrete obstacle problem. At each time step we solve the linear complementarity problem by policy iteration, following \cite{reisinger2012policy}. If $A_h=I-\Delta\tau\mathcal G_h$ and $b$ incorporates the preceding time step and boundary data, iteration stops when
\begin{equation*}
    \|\min\{A_hu-b,u-\phi\}\|_\infty<10^{-10},
\end{equation*}
where the minimum is componentwise. Nodes with $u_i-\phi_i>10^{-9}$ are classified as continuation nodes. We report the neighboring contact nodes as numerical boundaries. These are mesh estimates, and no smoothing of the computed curves is applied.

\subsection{Value and exercise boundaries}
Unless stated otherwise, the parameters are
\begin{equation*}
    r=0.05,\quad q=0.1,\quad\sigma=0.3,\quad
    K_1=1,\quad k_1=1.5,\quad k_2=2,\quad K_2=2.5,
    \quad 0\leq\tau\leq4.
\end{equation*}
The figures use $\Delta s=0.00125$ and $\Delta\tau=0.0005$. Table~\ref{tab_grid_convergence} compares coarser grids with the reference grid $(\Delta s,\Delta\tau)=(0.000625,0.00025)$. The value difference is the maximum over $s=0.05,0.055,\ldots,8$ and $\tau\in\{0.25,0.5,1,2,4\}$. The boundary difference is the maximum over all four spot boundaries at $\tau=0.01,0.02,\ldots,4$. Since no exact solution is available, these differences measure the numerical resolution only.

\begin{table}[H]
\centering
\begin{tabular}{cccc}
\hline
$\Delta s$ & $\Delta\tau$ & Maximum value difference & Maximum boundary difference\\
\hline
0.01000 & 0.00500 & $5.064\times10^{-4}$ & 0.03750\\
0.00500 & 0.00200 & $1.901\times10^{-4}$ & 0.01250\\
0.00250 & 0.00100 & $8.211\times10^{-5}$ & 0.00750\\
0.00125 & 0.00050 & $2.760\times10^{-5}$ & 0.00250\\
\hline
\end{tabular}
\caption{Grid refinement for $\sigma=0.3$, compared with the finer reference grid.}
\label{tab_grid_convergence}
\end{table}

For the largest volatility used below, $\sigma=0.4$, we also compare the domains $[0.05,12]$ and $[0.025,24]$ at the same mesh sizes $\Delta s=0.0025$ and $\Delta\tau=0.001$. The maximum value difference on the sample set is below $6\times10^{-14}$, and the reported boundaries agree to roundoff. In the three runs used for the figures, the maximum complementarity residual is below $1.9\times10^{-12}$. The obstacle and plateau bounds, time monotonicity and connectedness of each continuation interval are checked during the computation.

Figure~\ref{fig_value_profiles} displays the payoff and the value profiles, which increase with the time to expiry. The plateau stays fixed at $M=1.5$. Figure~\ref{fig_value_surface} shows the three-dimensional value surface together with the time-independent payoff and the four exercise boundaries. The value surface rises above the payoff in the two continuation regions and meets it along the boundary curves.

\begin{figure}[!htbp]
\centering
\includegraphics[width=\textwidth]{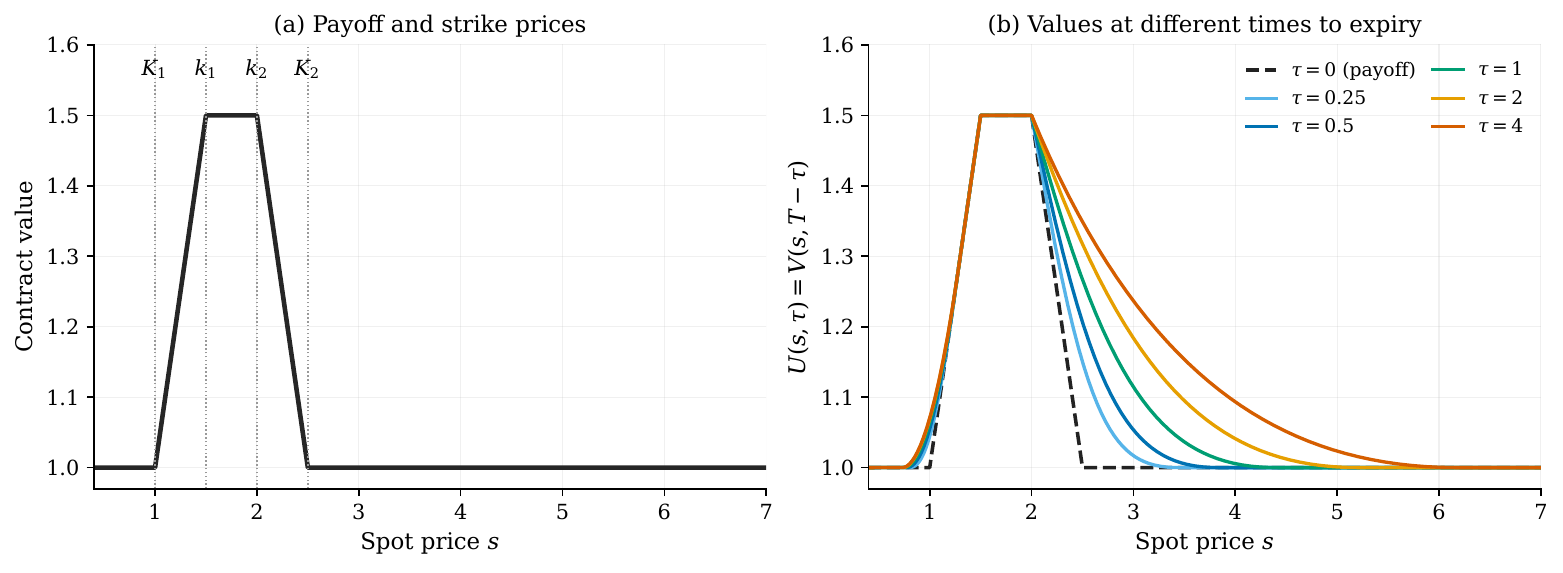}
\caption{Payoff and value profiles $U(s,\tau)$ for $\sigma=0.3$. The dashed profile in the right panel is the terminal payoff $U(s,0)=V_0(s)$.}
\label{fig_value_profiles}
\end{figure}

\begin{figure}[!htbp]
\centering
\includegraphics[width=\textwidth]{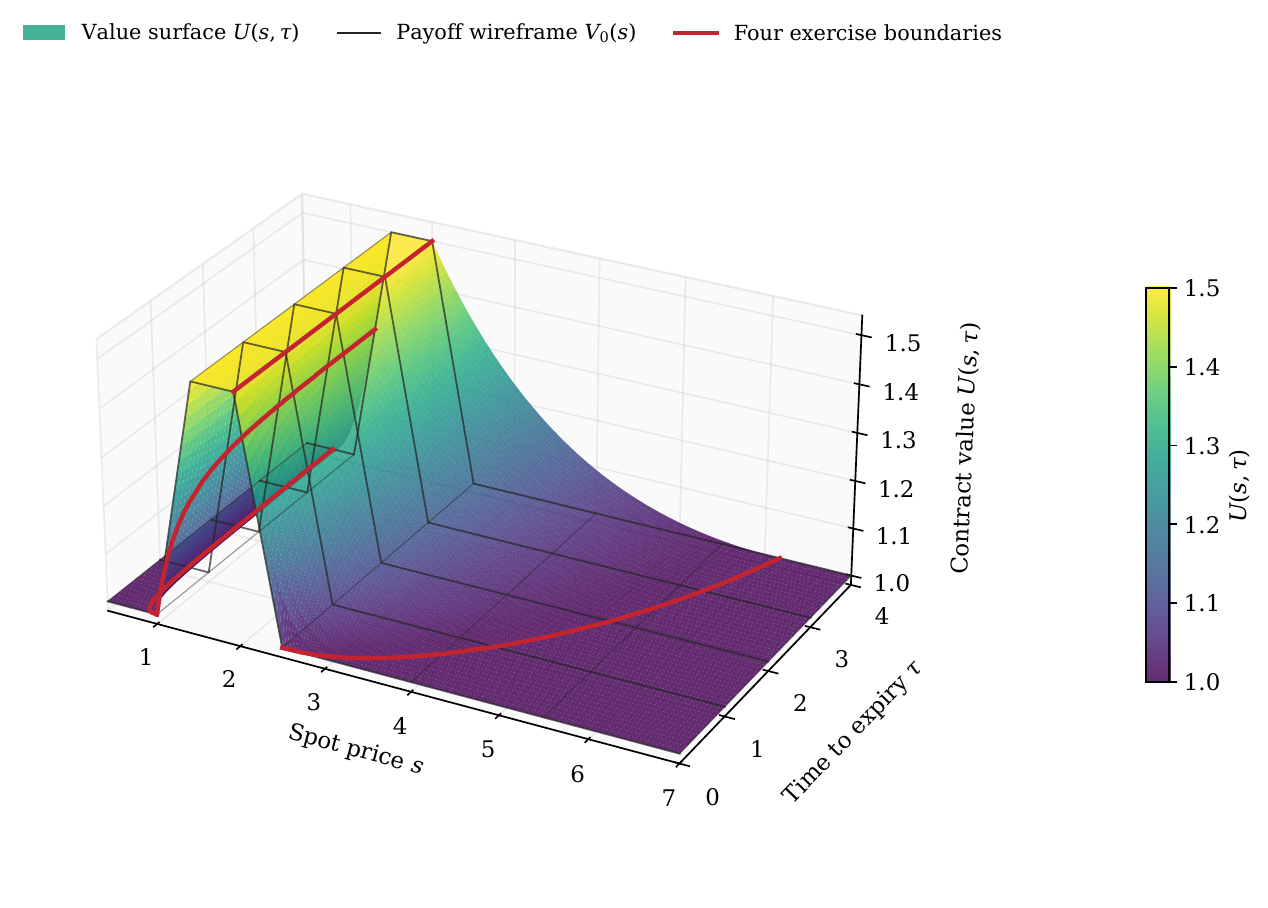}
\caption{Three-dimensional value surface $U(s,\tau)=V(s,T-\tau)$ for $\sigma=0.3$. The black wireframe is the payoff $V_0(s)$, and the four red curves are the spot exercise boundaries, drawn at height $V_0(s)$. The plateau has value $1.5$ for $1.5\leq s\leq2$. Boundary endpoints at $\tau=0$ represent the one-sided expiry limits.}
\label{fig_value_surface}
\end{figure}

Figure~\ref{fig_four_boundaries} shows the boundary curves directly. The continuation regions are
\begin{equation*}
    e^{A(\tau)}<s<e^{B(\tau)}
    \quad\text{and}\quad
    e^{\widetilde B(\tau)}<s<e^{\widetilde A(\tau)}.
\end{equation*}
The complement is the exercise region, including the whole plateau. In this example
\begin{equation*}
    \frac{r(K_2+k_2-K_1)}q=1.75<k_2,
\end{equation*}
so Theorem~\ref{b0_tilde} gives $e^{\widetilde B(\tau)}=k_2=2$ at every positive time. This boundary is therefore horizontal.

\begin{figure}[!htbp]
\centering
\includegraphics[width=0.88\textwidth]{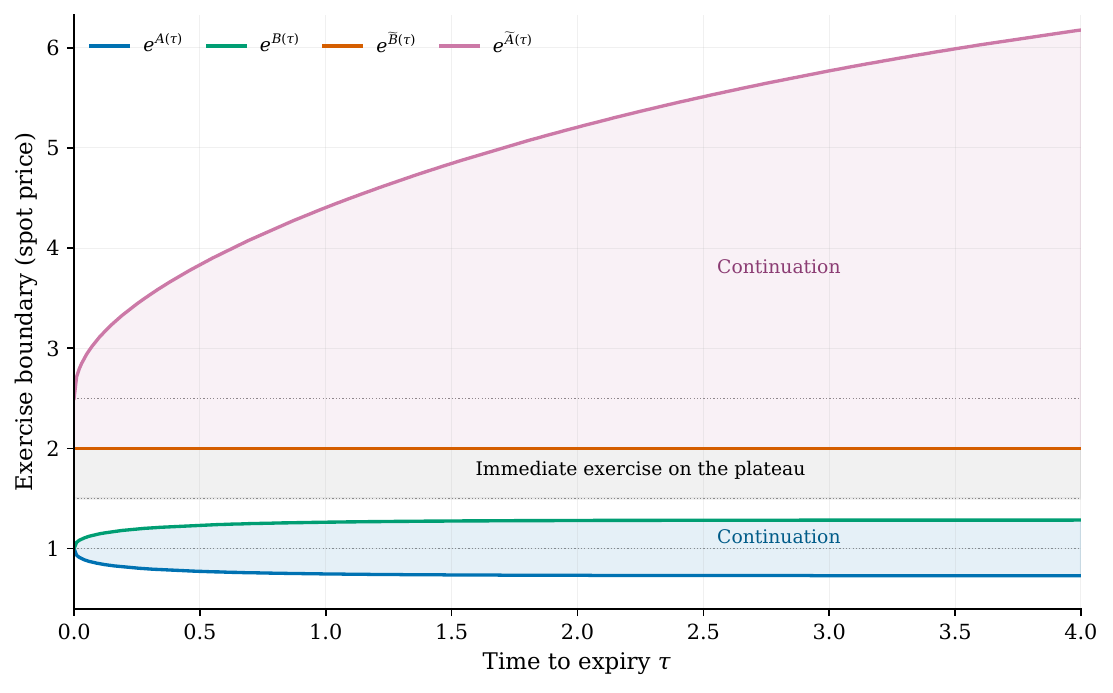}
\caption{The four exercise boundaries in spot coordinates for $\sigma=0.3$. The two colored bands are continuation regions, and the grey band is the plateau, where immediate exercise is optimal.}
\label{fig_four_boundaries}
\end{figure}

Finally, Figure~\ref{fig_volatility_boundaries} compares $\sigma\in\{0.2,0.3,0.4\}$ using the same spatial and temporal grids. Here $K_1+k_1-K_2=0$, so condition~\eqref{cond_convex_cap} holds. The decrease of $e^{A(\tau)}$ and increase of $e^{B(\tau)}$ with volatility agree with Theorem~\ref{thm:prop_freebdry}, and the latter boundary never exceeds $k_1$. The panels for the put-side boundaries are numerical observations for this parameter set. In particular, all three curves for $e^{\widetilde B(\tau)}$ coincide with $k_2$, as predicted above.

\begin{figure}[!htbp]
\centering
\includegraphics[width=\textwidth]{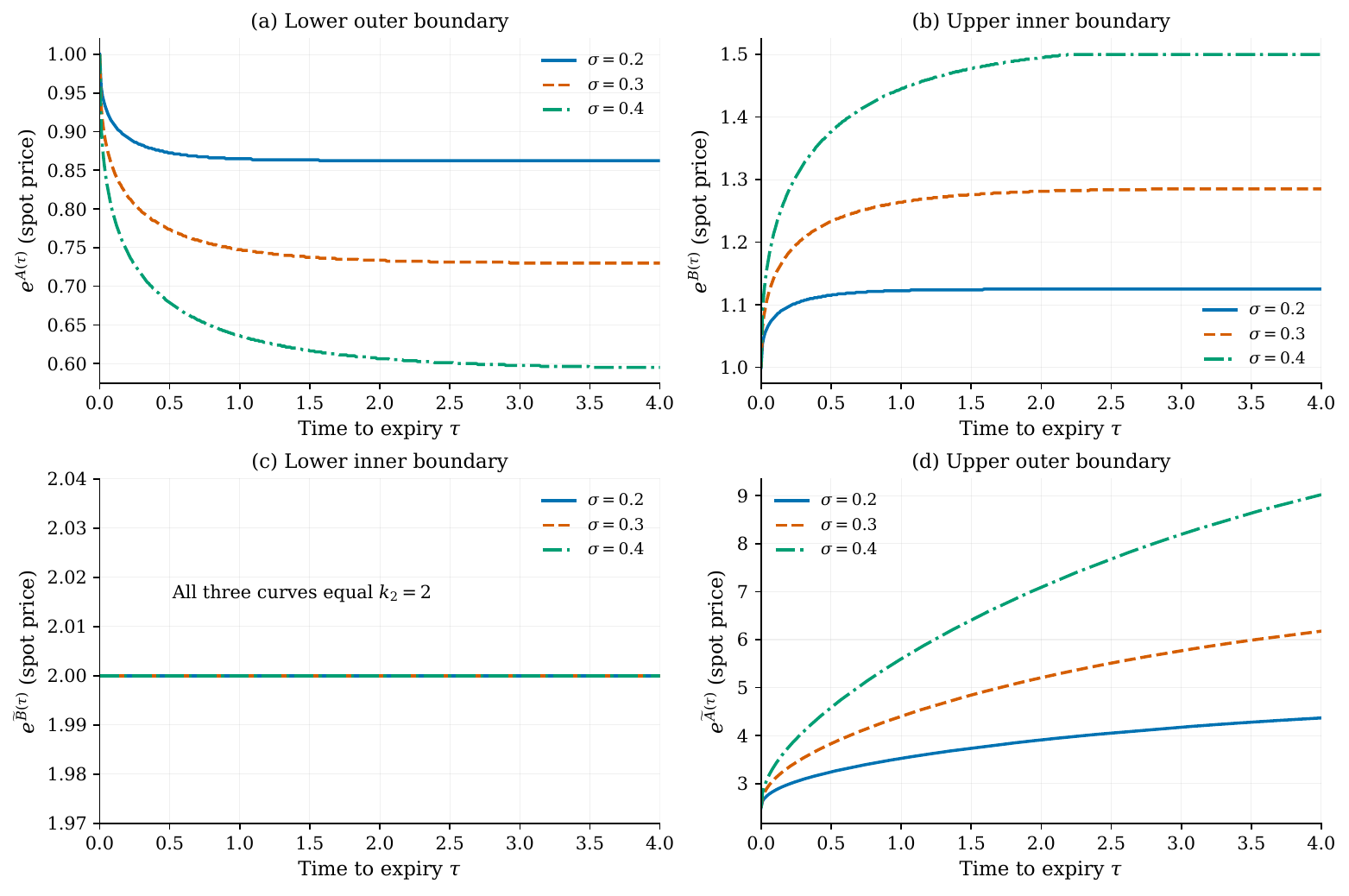}
\caption{Volatility comparison of the four spot boundaries. The horizontal axis is time to expiry. The three right inner boundary curves coincide at $k_2=2$, and the call-side inner boundary is capped at $k_1=1.5$.}
\label{fig_volatility_boundaries}
\end{figure}

\subsection{The convexity condition}
We next test Lemma~\ref{lem3.1} and the zero-dividend counterexample that follows it. For this example, set
\[
 (K_1,k_1,k_2,K_2)=(1,1.5,2,2.3),\qquad r=0.05,\qquad\sigma=0.3,
\]
so that $L:=k_1=1.5$, $D:=K_2-k_1=0.8$, $c:=K_1-D=0.2$ and $M:=K_2-K_1=1.3$,
and compare $q=0.015$ with $q=0$. The sufficient-condition margin $qL-rc$ is
respectively $0.0125$ and $-0.01$. We solve the call-side problem
\eqref{condor_spread_left} in the spot variable on $[0.04,L]$ with the
obstacle $(s-K_1)^++D$, using the coarse and fine $(\Delta s,\Delta\tau)$
pairs $(0.0025,2\times10^{-5})$ and $(0.00125,10^{-5})$ through $\tau=0.1$.
Figure~\ref{fig_convexity_condition_check} shows the resulting gamma, $U_{ss}$,
below the cap. For $q=0.015$, the interval shown in panel (a) lies in the
exercise region at $\tau=0.1$, and the retained smooth stencils have
nonnegative gamma up to roundoff throughout the domain, with negative values
of magnitude below $1.5\times10^{-10}$ at $\tau=0.1$. In contrast, for $q=0$
the gamma at $(s,\tau)=(1.48,0.1)$ is $-0.0827179$ on the coarse mesh and
$-0.0827197$ on the fine mesh. The fine-grid exercise premium there is
$2.05818\times10^{-4}$, placing the point inside continuation.
Halving the fine time step changes this gamma by less than
$4.25\times10^{-7}$, and expanding the left domain to $0.02$ changes it by
less than $1.5\times10^{-10}$.

Interior second differences are retained only when their three nodes lie
on the same smooth payoff branch and are all in continuation or all in
contact. Thus no retained interior stencil crosses the strike, a contact
transition, or the plateau endpoint. We also calculate the one-sided
cap gamma and compare it with the boundary PDE identity
\[
 U_{ss}(L-,\tau)
 =\frac{rM-(r-q)L U_s(L-,\tau)}{\frac12\sigma^2L^2},
\]
using one-sided approximations from below the cap. The identity is applied
only when continuation reaches the cap, as it does for $q=0$ here.
The small-time limit in the discussion following Lemma~\ref{lem3.1} is
$-2rc/(\sigma^2L^2)=-0.0987654$. At $\tau=0.0005$, the fine one-sided
second difference and the cap-identity estimate are $-0.0980720$ and
$-0.0981830$, respectively, and the coarser one-sided estimate is $-0.0970922$.
The refinement comparison in panel (b) makes the resolution of this
initial boundary layer explicit. These computations support the stated
sufficient condition and its zero-dividend counterexample, but they do not
replace the analytical proofs.

\begin{figure}[!htbp]
\centering
\includegraphics[width=0.85\textwidth]{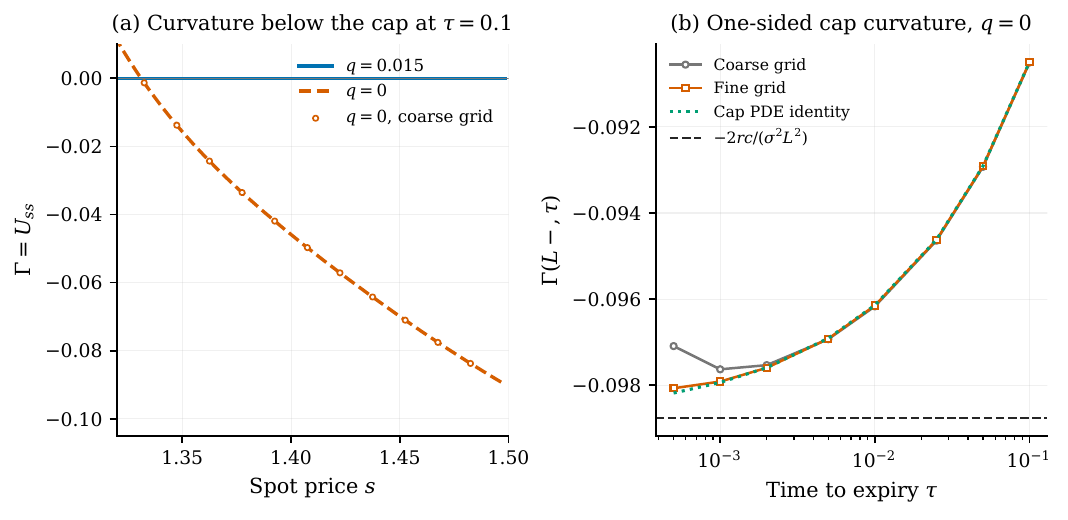}
\caption{The role of the sufficient condition $qL\geq rc$ for
$(K_1,k_1,k_2,K_2)=(1,1.5,2,2.3)$, $r=0.05$ and $\sigma=0.3$.
(a) Fine-grid gamma at $\tau=0.1$ below the cap for $q=0.015$ and $q=0$,
where coarse-grid markers provide a refinement comparison. Only valid smooth
interior stencils are shown. (b) One-sided cap gamma for $q=0$ on both
meshes, the independently evaluated cap PDE identity, and the analytical
small-time limit. No derivative is taken across the plateau corner.}
\label{fig_convexity_condition_check}
\end{figure}

\section*{Data availability}
No empirical data were used in this study, and all numerical parameters are illustrative. The code and data that reproduce the numerical results of Section~\ref{sec6} are available from the corresponding author upon reasonable request.

\bibliography{references}
\bibliographystyle{AIMS}

\end{document}